\documentclass[12pt]{amsart}
\usepackage{amsmath}
\usepackage{amsfonts, amssymb, euscript, mathrsfs}
\usepackage{amsthm, upref}
\usepackage{graphicx}
\usepackage[usenames, dvipsnames]{color}

\usepackage{tikz}
\usepackage{enumerate}
\usepackage[tmargin=0.8in,bmargin=0.8in,lmargin=0.6in,rmargin=0.6in]{geometry}
\usepackage{aliascnt}
\usepackage{hyperref}
\usepackage{verbatim}

\allowdisplaybreaks[4]

\newcommand{\MZ}{\mathbb{Z}}

\newcommand{\BR}{\mathbb{R}}

\newcommand{\SL}{\sum\limits}

\newcommand{\al}{\alpha}
\newcommand{\be}{\beta}

\newcommand{\de}{\delta}
\newcommand{\De}{\Delta}

\newcommand{\CF}{\mathcal F}

\newcommand{\MP}{\mathbf P}

\newcommand{\CS}{\mathcal S}

\newcommand{\CT}{\mathcal T}

\newcommand{\CP}{\mathcal P}

\newcommand{\Oa}{\Omega}

\newcommand{\si}{\sigma}

\newcommand{\pa}{\partial}
\renewcommand{\phi}{\varphi}

\newcommand{\eps}{\varepsilon}

\newcommand{\ol}{\overline}
\newcommand{\oR}{\overline{R}}
\newcommand{\om}{\overline{R}^{-1}}
\newcommand{\CM}{\mathcal M}

\newcommand{\norm}[1]{\lVert#1\rVert}
\renewcommand{\comment}[1]{}

\newcommand{\mP}{\mathbf{p}}

\newcommand{\md}{\mathrm{d}}

\newcommand{\mI}{\mathbf{1}}

\DeclareMathOperator{\SRBM}{SRBM}

\DeclareMathOperator{\diag}{diag}

\DeclareMathOperator{\tr}{tr}

\begin{document}

\theoremstyle{plain}
\newtheorem{thm}{Theorem}[section]
\newtheorem*{thmnonumber}{Theorem}
\newtheorem{lemma}[thm]{Lemma}
\newtheorem{prop}[thm]{Proposition}
\newtheorem{cor}[thm]{Corollary}
\newtheorem{open}[thm]{Open Problem}

\theoremstyle{definition}
\newtheorem{defn}{Definition}
\newtheorem{asmp}{Assumption}
\newtheorem{notn}{Notation}
\newtheorem{prb}{Problem}

\theoremstyle{remark}
\newtheorem{rmk}{Remark}
\newtheorem{exm}{Example}
\newtheorem{clm}{Claim}

\author{Cameron Bruggeman and Andrey Sarantsev}

\title[Multiple Collisions in Systems of Competing Brownian Particles]{Multiple Collisions  in Systems\\ of Competing Brownian Particles} 

\address{Department of Mathematics, Columbia University}

\email{bruggeman@math.columbia.edu}

\address{Department of Statistics and Applied Probability, University of California, Santa Barbara}

\email{sarantsev@pstat.ucsb.edu}

\date{May 23, 2016. Version 32}



\begin{abstract}
Consider a finite system of competing Brownian particles on the real line. Each particle moves as a Brownian motion, with drift and diffusion coefficients depending only on its current rank relative to the other particles. We find a sufficient condition for a.s. absence of a total collision (when all particles collide) and of other types of collisions, say of the three lowest-ranked particles. This continues the work of Ichiba, Karatzas, Shkolnikov (2013) and Sarantsev (2015). 
\end{abstract}

\maketitle

\section{Introduction} 

\subsection{A brief preview of results}

We start by  describing the concept of {\it competing Brownian particles} informally. A formal definition is postponed until the next section. 

Consider $N$ Brownian particles on the real line. Suppose the particle which is currently the $k$th leftmost one (we say: {\it has rank $k$}), moves as a Brownian motion with drift coefficient $g_k$ and diffusion coefficient $\si_k^2$. In other words, the behavior of a particle depends on its current rank relative to other particles. This is called a {\it system of competing Brownian particles}. 

A caveat: if two or more particles occupy the same position at the same time, how do we {\it resolve ties}, that is, assign ranks to these particles? We can use the following ``lexicographic" rule: particles $X_i$ with smaller indices $i$ get smaller ranks.

Let $X_k = (X_k(t), t \ge 0),\ k = 1, \ldots, N$, be these particles. Let $W_1, \ldots, W_N$ be i.i.d. Brownian motions. Then the particles $X_1, \ldots, X_N$ are governed by the following SDE:
\begin{equation}
\label{eq:SDE-CBP}
\md X_i(t) = \SL_{k=1}^N1\left(X_i\ \mbox{has rank}\ k\ \mbox{at time}\ t\right)\left(g_k\md t + \si_k\md W_i(t)\right).
\end{equation}
Let $Y_k(t)$ be the one of these $N$ particles which has rank $k$ at time $t$. The processes $X_i,\ i = 1, \ldots, N$, are called {\it named particles}, and $Y_k,\ k = 1, \ldots, N$, are called {\it ranked particles}. If $X_i(t) = Y_k(t)$, we say that the corresponding particle at time $t$ has {\it name} $i$ and {\it rank} $k$. By definition, the ranked particles satisfy
\begin{equation}
\label{154}
Y_1(t) \le Y_2(t) \le \ldots \le Y_N(t),\ \ t \ge 0.
\end{equation}
Weak existence and uniqueness in law for these systems were proved in \cite{Bass1987}.
Some motivation for studying these systems is provided later in the Introduction.

This article is devoted to {\it collisions of competing Brownian particles}. 
Let us exhibit some results proved in this paper; they are corollaries of general theorems from Section 2. Suppose, for the sake of simplicity, that we have $N = 4$ competing Brownian particles. We shall present some results, and explain how they are related to the paper \cite{MyOwn3}. 

\begin{thm}
\label{example1}
If the following conditions  
$$
\begin{cases}
9\si_1^2 \le 7\si_2^2 + 7\si_3^2 + 7\si_4^2;\\
3\si_1^2  \le 5\si_2^2 + \si_3^2 + \si_4^2;\\
3\si_1^2 + 3\si_4^2  \le 5\si_2^2 + 5\si_3^2;\\
3\si_4^2 \le \si_1^2 + \si_2^2 + 5\si_3^2;\\
9\si_4^2 \le 7\si_1^2 + 7\si_2^2 + 7\si_3^2,
\end{cases}
$$
hold, then a.s. there does not exist $t > 0$ such that 
\begin{equation}
\label{type1}
Y_1(t) = Y_2(t) = Y_3(t) = Y_4(t).
\end{equation}
Moreover, a.s. there does not exist $t > 0$ such that
\begin{equation}
\label{type2}
Y_1(t) = Y_2(t)\ \ \mbox{and}\ \ Y_3(t) = Y_4(t).
\end{equation}
\end{thm}

\begin{thm}
\label{example3}
If the five inequalities from Theorem~\ref{example1} together with 
$$
\si_2^2 \ge \frac12\left(\si_1^2 + \si_3^2\right)
$$
hold, then a.s. there does not exist $t > 0$ such that 
\begin{equation}
\label{type3}
Y_1(t) = Y_2(t) = Y_3(t).
\end{equation}
\end{thm}

Similar statements (but with other inequalities involving $\si_k^2,\ k = 1, \ldots, N$) can be stated for any $N = 5, 6, \ldots$, and for any type of collision between $Y_1, \ldots, Y_N$. We also have the following statement for $N = 4$; however, in this case we did not find any specific generalizations for this result in cases $N \ge 5$. 

\begin{thm}
\label{cams}
With $N = 4$, if the following condition holds:
\begin{equation}
\label{Bruggeman}
\si_1^2 + \si_4^2 \le \si_2^2 +\si_3^2,
\end{equation}
then a.s. there are no $t > 0$ such that~\eqref{type1} and~\eqref{type2} hold. 
\end{thm}



\subsection{Relation to previous results from the paper \cite{MyOwn3}} The results of this paper complement the main result of the companion paper \cite{MyOwn3}. Let us discuss the relation between these two papers. 

\begin{defn}
A {\it triple collision at time $t$} occurs if there exists a rank $k = 2, \ldots, N-1$ such that  $Y_{k-1}(t) = Y_{k}(t) = Y_{k+1}(t)$. A {\it simultaneous collision} at time $t$ occurs if there are ranks $k \ne l$ such that such that $Y_{k}(t) = Y_{k+1}(t),\ Y_{l}(t) = Y_{l+1}(t)$. 
\label{triplesimdef}
\end{defn}

Note that a triple collision is a particular case of a simultaneous collision. One motivation for  studying triple collisions is that a strong solution to SDE~\eqref{classicSDE} is known to exist and be unique up to the first moment of a triple collision: this was proved in \cite{IKS2013}. The question of whether a classical system of competing Brownian particles a.s. avoids  triple collisions was studied in \cite{IK2010, IKS2013}, with significant partial results obtained. In our companion paper \cite{MyOwn3}, the following necessary and sufficient condition was found. 

\begin{prop}
\label{elegant}
A system of $N$ competing Brownian particles has a.s. no triple collisions and no simultaneous collisions at any time $t > 0$, if and only if the sequence $(\si_1^2, \ldots, \si_N^2)$ is concave, that is, 
\begin{equation}
\label{concave}
\frac12\left(\si_{k-1}^2 + \si_{k+1}^2\right) \le \si_k^2,\ \ k = 2, \ldots, N-1.
\end{equation}
If the condition~\eqref{concave} is violated for some $k = 2, \ldots, N-1$, then with positive probability there exists $t > 0$ such that $Y_{k-1}(t) = Y_k(t) = Y_{k+1}(t)$. 
\end{prop}

An interesting corollary: {\it If there are a.s. no triple collisions at any time $t > 0$, then there are a.s. no simultaneous collisions at any time $t > 0$. }

We call the condition~\eqref{concave} {\it global concavity}, as opposed to {\it local concavity at rank $j$}, with just one inequality:
$$
\frac12\left(\si_{j-1}^2 + \si_{j+1}^2\right) \le \si_j^2. 
$$
Thus,  if there is no local concavity at $k$, then with positive probability there is a triple collision between $Y_{k-1}$, $Y_k$ and $Y_{k+1}$. 
However, we do not know whether the converse is true: if there is local concavity at $k$, then there are a.s. no triple collisions between $Y_{k-1}$, $Y_k$, $Y_{k+1}$. 

Proposition~\ref{elegant} is a condition to avoid {\it all} possible triple collisions. If we are interested in avoiding only an individual triple collision, such as
$$
Y_1(t) = Y_2(t) = Y_3(t),
$$
we can get another sufficient condition for this: see Theorem~\ref{example3} above. This condition is not stronger than~\eqref{concave}: we can find diffusion parameters, say 
$$
\si_1^2 = \si_2^2 = \si_4^2 = 1,\ \si_3^2 = 0.9,
$$
which satisfy the conditions in Theorem~\ref{example3}, but do not satisfy~\eqref{concave}. In this case, there are no triple collisions of the type
$$
Y_1(t) = Y_2(t) = Y_3(t),
$$
and no simultaneous collisions of the type
$$
Y_1(t) = Y_2(t),\ \ Y_3(t) = Y_4(t),
$$
but with positive probability there are triple collisions of the type
$$
Y_2(t) = Y_3(t) = Y_4(t).
$$
We can take 
$$
\si_1^2 = \si_4^2 = 1,\ \si_2^2 = \si_3^2 = 0.9.
$$
Then   local concavity fails at ranks $2$ and   $3$. Therefore, with positive probability there exists a triple collision between ranked particles $Y_1$, $Y_2$, and $Y_3$, and with positive probability there exists a triple collision between ranked particles $Y_2$, $Y_3$, and $Y_4$. However, the five inequalities~\eqref{CP4I} are satisfied. Therefore, there are no simultaneous collisions of the type 
$$
Y_1(t) = Y_2(t),\ Y_3(t) = Y_4(t).
$$
It was noted in \cite[Corollary 1.3]{MyOwn3} that if there are a.s. no triple collisions, then there are a.s. no simultaneous collisions. As we see in this example, it is possible to find  diffusion coefficients so that the system avoids simultaneous collisions of the type~\eqref{type2}, but exhibits triple collisions with positive probability. 

Also, the collision as in~\eqref{type1} is stronger than a triple or a simultaneous collision.

\subsection{Outline of the proofs}

The main results of this paper are Theorems~\ref{totalcor} and~\ref{mainthm}. Theorems~\ref{example1} and~\ref{example3}, along with other examples in Section 2, are corollaries of these two results. Theorems~\ref{totalcor} and~\ref{mainthm} are proved  in Sections 3 and 4. Let us give a brief outline of the proofs. 

Consider the gaps between the consecutive ranked particles:
$$
Z_1(t) = Y_2(t) - Y_1(t), \ldots, Z_{N-1}(t) = Y_N(t) - Y_{N-1}(t), \quad 0 \le t <\infty.
$$
These  form an $(N-1)$-dimensional process in $\BR^{N-1}_+$, which is called the {\it gap process} and is denoted by $Z = (Z(t), t \ge 0)$. It turns out that $Z$ is a particular case of a well-known process, which is called a {\it semimartingale reflected Brownian motion (SRBM)} in a positive multidimensional orthant. We discuss this relationship in subsection 4.2.

Let us informally describe the concept of an SRBM; a formal definition is given in subsection 3.1. Fix the dimension $d \ge 1$, and let $\BR_+ := [0, \infty)$ be the positive half-axis. Let $S = \BR^d_+$ be the $d$-dimensional positive orthant. Fix a {\it drift vector} $\mu \in \BR^d$, a $d\times d$ {\it reflection matrix} $R$ and another $d\times d$ {\it covariance matrix} $A$. An {\it SRBM in the orthant $S$} with these parameters $R, \mu, A$ is a Markov process which:

(i) behaves as a $d$-dimensional Brownian motion with drift vector $\mu$ and covariance matrix $A$ in the interior of the orthant $S$;

(ii) at each face $S_i := \{x \in S\mid x_i = 0\}$, $i = 1, \ldots, d$, this process is reflected in the direction of $r_i$, which is the $i$th column of $R$. 

If $r_i = e_i$, where $e_i$ is the $i$th vector from the standard basis in $\BR^d$, then this reflection is called {\it normal}; otherwise, it is called {\it oblique}. This process is denoted by $\SRBM^d(R, \mu, A)$. The survey \cite{Wil1995} provides a good overview of this process. More information and citations concerning an SRBM are provided in subsection 3.1. 

The parameters $R$, $\mu$ and $A$ of the SRBM which is the gap process depend on $g_n, \si_n,\ n = 1, \ldots, N$, see subsection 4.2, equations ~\eqref{R12}, ~\eqref{A} and~\eqref{mu}. 

\medskip

Let us return to the examples above. Consider a system of $N = 4$ competing Brownian particles. A collision of the type~\eqref{type1} is equivalent to 
$Z_1(t) = Z_2(t) = Z_3(t) = 0$, that is, to the process $Z$ hitting the corner of the orthant $\BR^3_+$ at time $t$. A collision of the type~\eqref{type2} is equivalent to 
$Z_1(t) = Z_3(t) = 0$, that is, to the process of gaps $Z$ hitting the {\it edge} $\{x \in \BR^3_+\mid x_1 = x_3 = 0\}$ of the boundary $\pa \BR^3_+$ of the orthant $\BR^3_+$. Similarly, we can rewrite other types of collisions in terms of the gap process.  

In Section 3, we obtain results concerning an SRBM a.s.$\,$avoiding corners or edges. In Theorem~\ref{cornerthm}, we find a sufficient condition for an $\SRBM^d(R, \mu, A)$ to a.s. avoid the corner of the orthant $S = \BR^d_+$. In Theorem~\ref{corner2edge}, we find a sufficient condition for an $\SRBM^d(R, \mu, A)$ to a.s. avoid the ``edge" 
$$
S_I := \{x \in S\mid x_i = 0,\ i \in I\}
$$
of a given subset $I \subseteq \{1, \ldots, d\}$. This is done by reducing the property  of avoiding an edge to the property of avoiding a corner, and  allows us to find a sufficient condition for a.s. avoiding an edge $S_I$; see Corollary~\ref{general}. 

In Section 4, we find the relationship between the parameters of an SRBM and the parameters of the system of competing Brownian particles, see~\eqref{R12}, ~\eqref{mu} and~\eqref{A}. Then we apply results of Section 3 to finish the proofs of Theorem~\ref{totalcor} and~\ref{mainthm}. 

\subsection{Motivation} Systems of competing Brownian particles are used in Stochastic Finance: the process
\begin{equation}
\label{rankbasedmarketmodel}
\bigl(e^{X_1(t)}, \ldots, e^{X_N(t)}\bigr)'
\end{equation}
can be viewed as a stock market model, see \cite{BFK2005}. Here, $e^{X_i(t)}$ is the capitalization of the $i$th stock at time $t \ge 0$. In real world, stocks with smaller capitalizations have the following property: logarithms of their capitalizations have larger drift coefficients (which in financial terminology are called {\it growth rates}) and larger diffusion coefficients (which are called {\it volatilities}) than that of stocks with larger capitalizations. It is easy to construct a model~\eqref{rankbasedmarketmodel} which captures this property: just let 
$$
g_1 > g_2 > \ldots > g_N\ \ \mbox{and}\ \ \si_1 > \si_2 > \ldots > \si_N.
$$
For  applications to financial market models similar to~\eqref{rankbasedmarketmodel}, see the articles \cite{Ichiba11, FIK2013b, CP2010, JR2013b, MyOwn4}, the book \cite[Chapter 5]{F2002} and the somewhat more recent survey \cite[Chapter 3]{FK2009}. 

These systems also arise as discrete analogues of a so-called {\it nonlinear diffusion process}, governed by {\it McKean-Vlasov equation}, studied in \cite{Chaos1, Chaos2, Chaos3, Dawson}. As $N \to \infty$, systems of competing Brownian particles converge weakly (in a certain sense) to a nonlinear diffusion process, see \cite{S2012, JR2013a,4people}. 

Also, let us mention that systems of competing Brownian particles serve as scaling limits of a certain type of exclusion processes on $\MZ$, namely asymmetrically colliding random walks, see \cite{KPS2012}. 

Systems of competing Brownian particles were also studied in the papers \cite{Ichiba11, PS2010, PP2008, CP2010, IPS2012, IKS2013, IK2010, IchibaThesis, FIK2013, Reygner2014, JR2014}. 

\subsection{Organization of the paper} Section 2 contains rigorous definitions, main results: Theorems~\ref{totalcor} and~\ref{mainthm}, and examples (including the ones mentioned above). Section 3 is devoted to a  semimartingale reflected Brownian motion in the orthant,  and contains conditions for it to avoid edges of the boundary of this orthant. Section 4 applies results of Section 3 to systems of competing Brownian particles. Proofs of Theorems~\ref{totalcor}, ~\ref{mainthm} and~\ref{cams} are contained there.
Section 5 deals with a generalization of the concept of competing Brownian particles: the so-called {\it systems with asymmetric collisions}. The Appendix contains a few technical lemmas. 

\section{Formal Definitions and Main Results}

\subsection{Notation}  We denote by $I_k$ the $k\times k$-identity matrix. For a vector $x = (x_1, \ldots, x_d)' \in \BR^d$, let $\norm{x} := \left(x_1^2 + \ldots + x_d^2\right)^{1/2}$ be its Euclidean norm. 
For any two vectors $x, y \in \BR^d$, their dot product is denoted by $x\cdot y = x_1y_1 + \ldots + x_dy_d$. We compare vectors $x$ and $y$ componentwise: $x \le y$ if $x_i \le y_i$ for all $i = 1, \ldots, d$; $x < y$ if $x_i < y_i$ for all $i = 1, \ldots, d$; similarly for $x \ge y$ and $x > y$. We compare matrices of the same size componentwise, too. For example, we write $x \ge 0$ for $x \in \BR^d$ if $x_i \ge 0$ for $i = 1, \ldots, d$; $C = (c_{ij})_{1 \le i, j \le d} \ge 0$ if $c_{ij} \ge 0$ for all $i$, $j$. The symbol $a'$ denotes the transpose of (a vector or a matrix) $a$. 

Fix $d \ge 1$, and let $I \subseteq \{1, \ldots, d\}$ be a nonempty subset. Write its elements in increasing order: $I = \{i_1, \ldots, i_m\},\ \ 1 \le i_1 < i_2 < \ldots < i_m \le d$. For any $x \in \BR^d$, let
$[x]_I := (x_{i_1}, \ldots, x_{i_m})'$. For any $d\times d$-matrix $C = (c_{ij})_{1 \le i, j \le d}$, let $[C]_I := \left(c_{i_ki_l}\right)_{1 \le k, l \le m}$. We let $\mI := (1, \ldots, 1)'$ (the dimension of this vector depends on the context).

\subsection{Definitions}

Now, let us define systems of competing Brownian particles formally. Assume we have the usual setting: a filtered probability space $(\Oa, \CF, (\CF_t)_{t \ge 0}, \MP)$ with the filtration satisfying the usual conditions.  The term {\it standard Brownian motion} stands for a one-dimensional Brownian motion with drift coefficient zero and diffusion coefficient one, starting from zero. 

\begin{defn}
Consider a continuous adapted $\BR^N$-valued process 
$$
X = (X(t), t \ge 0),\ \ X(t) = (X_1(t), \ldots, X_N(t))'.
$$
For every $t \ge 0$, let $\mP_t$ be the permutation of $\{1, \ldots, N\}$ which:

\noindent
(i) {\it ranks the components of} $X(t)$, that is, $X_{\mP_t(i)}(t) \le X_{\mP_t(j)}(t)$ for $1 \le i < j \le N$;

\noindent
(ii) {\it resolves ties in   lexicographic order:} if $X_{\mP_t(i)}(t) = X_{\mP_t(j)}(t)$ and $i < j$, then $\mP_t(i) < \mP_t(j)$. 

\smallskip
Fix parameters $g_1, \ldots, g_N \in \BR$ and $\si_1, \ldots, \si_N > 0$, and let $W_1, \ldots, W_N$ be i.i.d. standard $(\CF_t)_{t \ge 0}$-Brownian motions. 

Suppose the process $X$ satisfies the following SDE:
\begin{equation}
\label{classicSDE}
\md X_i(t) = \SL_{k=1}^N1(\mP_t(k) = i)\left[g_k\md t + \si_k\md W_i(t)\right],\ \ i = 1, \ldots, N.
\end{equation}
Then $X$ is called a {\it classical system of $N$ competing Brownian particles}. For $k = 1, \ldots, N$, the process
$$
Y_k = (Y_k(t), t \ge 0),\ \ Y_k(t) \equiv X_{\mP_t(k)}(t)
$$
is called the {\it $k$th ranked particle}. 
\label{classicdef}
\end{defn}

We use the term {\it classical} to distinguish these systems from similar systems of competing Brownian particles with so-called {\it asymmetric collisions}; more on this in Section 4. 

\begin{defn}
Consider a system from Definition~\ref{classicdef}. We say that a {\it collision of order $M$ occurs at time $t \ge 0$,} if there exists $k = 1, \ldots, N$ such that  
$$
Y_{k}(t) = Y_{k+1}(t) = \ldots = Y_{k+M}(t).
$$
A collision of order $M = 2$ is called a {\it triple collision}. A collision of order $M = N-1$ is called a {\it total collision}. 
\end{defn}

As mentioned before, a related example of a total collision (for a slightly different SDE) was considered in the paper \cite{Bass1987}. 

There is another closely related concept. We can have, for example, $Y_1(t) = Y_2(t)$ and $Y_4(t) = Y_5(t) = Y_6(t)$ at the same moment $t \ge 0$. This is called a {\it multicollision} of a certain order (this particular one is of order $3$). 

\begin{defn} Consider a system from Definition~\ref{classicdef}, and fix a nonempty subset $I \subseteq \{1, \ldots, N-1\}$. A {\it multicollision with pattern} $I$ occurs at time $t \ge 0$ if 
$$
Y_k(t) = Y_{k+1}(t),\ \ \mbox{for all}\ \ k \in I.
$$
We shall sometimes say that {\it there are no multicollisions with pattern $I$} if a.s. there does not exist $t > 0$ such that there is a multicollision with pattern $I$ at time $t$. 
\label{Pat}
\end{defn}

A multicollision with pattern $I$ has order $M = |I|$. If $I = \{k, k+1, \ldots, l-2, l-1\}$, then a multicollision with pattern $I$ is, in fact, a multiple collision of particles with ranks $k, k+1, \ldots, l-1, l$. If $I = \{1, \ldots, N-1\}$, this is a total collision. If $I = \{k, l\}$, this is a simultaneous collision. If $I = \{k, k+1\}$, this is a triple collision. 

We can immediately state some results which reduce multicollisions to total collisions. 

\begin{lemma}  
 Fix $1 < N_1 \le N_2 < N$. Suppose that $\si_1, \ldots, \si_N \ge 0$ are such that for a system of competing Brownian particles with parameters $\si_{N_1}, \ldots, \si_{N_2}$ a multicollision with pattern $I \subseteq \{N_1, \ldots, N_2\}$ happens with positive probability. Then for a system of competing Brownian particles with parameters $\si_1,\ldots, \si_N$ this multicollision also happens with positive probability. 
\end{lemma}

\begin{proof} This follows from the relation between multicollisions and hitting edges of $\BR^{N-1}_+$ by the gap process, established in Lemma~\ref{red}, and from Theorem~\ref{corner2edge2}. 
\end{proof}

It is worth providing some references about a diffusion hitting a lower-dimensional manifold: the articles \cite{Friedman1974, Ramasubramanian1983, Ramasubramanian1988, CepaLepingle}, and the book \cite{FriedmanBook}.

\medskip
In this paper, we are interested in triple and simultaneous collisions, as well as the collisions of higher order $M \ge 4$. We examine whether the classical system of competing Brownian particles avoid collisions (and multicollisions) with given pattern. This paper contains two main results.  One is a sufficient condition for absence of   total collisions. The other is more general: a sufficient condition for the absence of   multicollisions with a given pattern. The approach taken in this article does not give   necessary and sufficient conditions for absence of multicollisions, only sufficient conditions; neither does it provide  conditions for having multicollisions with positive probability (as opposed to avoiding them).

\subsection{Avoiding a multicollision depends only on diffusion coefficients}

The following lemma tells us that the property of a system of competing Brownian particles to avoid multicollisions with a given pattern is independent of the initial conditions $x$ and the drift coefficients $g_1, \ldots, g_N$. In other words, it can possibly depend only on the diffusion coefficients $\si_1^2, \ldots, \si_N^2$. 

\begin{lemma} Take a classical system of competing Brownian particles from Definition~\ref{classicdef}. Fix $I \subseteq \{1, \ldots, N-1\}$, a pattern. Let $x \in \BR^N$ be the initial conditions, and let $\MP_x$ be the corresponding probability measure. Denote by 
\begin{equation}
\label{probb}
p\left(g_1,\, g_2,\, \ldots,\, g_N,\, \si_1,\, \si_2,\, \ldots\, \si_N,\, x\right)
\end{equation}
the probability that there exists a moment $t > 0$ such that the system, starting from $x$, will experience a multicollision with pattern $I$ at this moment. For fixed $\si_1, \ldots, \si_N > 0$, either 
$$
p\left(g_1,\, g_2,\, \ldots,\, g_N,\, \si_1,\, \si_2,\, \ldots\, \si_N,\, x\right) = 0\ \ \mbox{for all}\ \ x \in \BR^N,\ \ (g_k)_{1 \le k \le N} \in \BR^N,
$$
or
$$
p\left(g_1,\, g_2,\, \ldots,\, g_N,\, \si_1,\, \si_2,\, \ldots\, \si_N,\, x\right) > 0\ \ \mbox{for all}\ \ x \in \BR^N,\ \ (g_k)_{1 \le k \le N} \in \BR^N.
$$
\label{Indep}
\end{lemma}

However, in the second case (when the probability~\eqref{probb} is positive) the exact value of this probability depends on the initial conditions $x$ and the drift coefficients $g_1, \ldots, g_N$. This follows from Remark 5 in \cite[Subsection 3.2]{MyOwn3} and connection between competing Brownian particles and an SRBM, discussed just above. 
The proof is postponed until Appendix.

\subsection{Sufficient conditions for avoiding total collisions}

Let us introduce some additional notation. Let $M \ge 2$. For
$$
\al = (\al_1, \ldots, \al_M)' \in \BR^M\ \ \mbox{and}\ \ l = 1, \ldots, M-1,
$$
we define
$$
c_{l}(\al) := -\frac{2(M-1)}{M}\al_1^2 + \frac{2(M+1)}M\SL_{p=2}^l\al_p^2 + \frac{2(M-1)(M-l) - 4l}{(M-l)M}\SL_{p=l+1}^M\al_p^2.
$$
We also denote by $\al^{\leftarrow} := (\al_M, \ldots, \al_1)'$ the vector $\al$ with components put in the reverse order. Note that $c_{M-1}(\al) =  c_{M-1}\left(\al^{\leftarrow}\right)$. Let 
\begin{equation}
\label{CP}
\CP(\al) := \min\left(c_1(\al),\, c_1\left(\al^{\leftarrow}\right),\, c_2(\al),\, c_2\left(\al^{\leftarrow}\right),\, \ldots,\, c_{M-2}(\al),\, c_{M-2}\left(\al^{\leftarrow}\right),\, c_{M-1}(\al)\right).
\end{equation}
For example, in cases $M = 2$ and $M = 3$ we have the following expressions for $\CP(\al)$:
\begin{equation}
\label{CP2}
\CP(\al_1, \al_2) = c_1(\al_1, \al_2) = -\al_1^2 - \al_2^2\,,
\end{equation}
\begin{equation}
\label{CP3}
\CP(\al_1, \al_2, \al_3) = \min\left(\frac83\al_2^2 - \frac43\al_1^2 - \frac43\al_3^2,\ \  \frac23\al_2^2 + \frac23\al_3^2 - \frac43\al_1^2,\ \  \frac23\al_1^2 + \frac23\al_2^2 - \frac43\al_3^2\right).
\end{equation}

\begin{thm}
\label{totalcor}
Consider a classical system of competing Brownian particles from Definition~\ref{classicdef}, and denote $$\si := (\si_1, \ldots, \si_N)' .$$ If $\,\CP(\si) \ge 0$ in the notation of \eqref{CP}, then a.s.   there is no total collision at any time $t>0$. 
\end{thm}

\subsection{Examples of avoiding total collisions} In this subsection, we consider systems of $N = 3$, $N = 4$ and $N = 5$ particles. We apply Theorem~\ref{totalcor} to find a sufficient condition for a.s. avoiding total collisions. In particular, we compare our results for three particles to a necessary and sufficient condition~\eqref{concave}. We also compare results for $N = 4$ particles given by Theorem~\ref{totalcor} and Theorem~\ref{cams}. 

\begin{exm} {\it The case of $N = 3$ particles.} In this case, ``triple collision'' is a synonym for ``total collision''. The quantity $\CP(\si)$ is calcluated in~\eqref{CP3}. The inequality $\CP(\si) \ge 0$ is equivalent to the following system:
\begin{equation}
\label{CP3I}
\begin{cases}
\si_1^2 + \si_3^2 \le 2\si_2^2;\\
2\si_1^2  \le \si_2^2 + \si_3^2;\\
2\si_3^2  \le \si_2^2 + \si_1^2.
\end{cases}
\end{equation}
In fact, the first inequality in~\eqref{CP3I} follows from the second and the third ones. Therefore,~\eqref{CP3I} is equivalent to
\begin{equation}
\label{CP3II}
\begin{cases}
2\si_1^2  \le \si_2^2 + \si_3^2;\\
2\si_3^2  \le \si_2^2 + \si_1^2.
\end{cases}
\end{equation}
This sufficient condition is more restrictive than~\eqref{concave}, which for $N = 3$ particles takes the form $2\si_2^2 \ge \si_1^2 + \si_3^2$. Therefore, Theorem~\ref{totalcor} gives a weaker result than the result from \cite{MyOwn3}, mentioned in Proposition~\ref{elegant}. In other words, for three particles the results from this paper do not give us anything new compared to \cite{MyOwn3, IK2010}, which is not surprising: in \cite{MyOwn3, IK2010}, they found a necessary and sufficient condition for avoiding total collision for $N = 3$ particles.  
\label{N3}
\end{exm}

\begin{exm} {\it The case of $N = 4$ particles.} The result was stated in the Introduction as Theorem~\ref{example1}. The condition 
$\CP(\si) \ge 0$ holds,  if and only if all the following five inequalities hold: 
\begin{equation}
\label{CP4I}
\begin{cases}
9\si_1^2 \le 7\si_2^2 + 7\si_3^2 + 7\si_4^2;\\
3\si_1^2  \le 5\si_2^2 + \si_3^2 + \si_4^2;\\
3\si_1^2 + 3\si_4^2  \le 5\si_2^2 + 5\si_3^2;\\
3\si_4^2 \le \si_1^2 + \si_2^2 + 5\si_3^2;\\
9\si_4^2 \le 7\si_1^2 + 7\si_2^2 + 7\si_3^2.
\end{cases}
\end{equation}

As mentioned in Section 1, let $\si_1^2 = \si_2^2 = \si_4^2 = 1$, and $\si_3^2 = 0.9$. Then there are triple collisions between the particles $Y_2$, $Y_3$ and $Y_4$ with positive probability, because the sequence $(\si_1^2, \si_2^2, \si_3^2, \si_4^2)$ is not concave: it does not satisfy the condition~\eqref{concave}. But the condition $\CP(\si) \ge 0$ is satisfied, hence there are a.s. no total collisions.  Note that this example satisfies the conditions of Theorem~\ref{totalcor}, but fails to satisfy those of Theorem~\ref{cams}.
\label{N4}
\end{exm}

\begin{exm} {\it The case of $N = 5$ particles.} In this case  $\CP(\si) \ge 0$ is equivalent to the following seven inequalities:
\begin{equation}
\label{CP5I}
\begin{cases}
8\si_1^2  \le 7\si_2^2 + 7\si_3^2 + 7\si_4^2 + 7\si_5^2;\\
6\si_1^2 \le 9\si_2^2 + 4\si_3^2 + 4\si_4^2 + 4\si_5^2;\\
4\si_1^2  \le 6\si_2^2 + 6\si_3^2 + \si_4^2 + \si_5^2;\\
2\si_1^2 + 2\si_5^2  \le 3\si_2^2 + 3\si_3^2 + 3\si_4^2;\\
8\si_5^2  \le 7\si_4^2 + 7\si_3^2 + 7\si_2^2 + 7\si_1^2;\\
6\si_5^2 \le 9\si_4^2 + 4\si_3^2 + 4\si_2^2 + 4\si_1^2;\\
4\si_5^2  \le 6\si_4^2 + 6\si_3^2 + \si_2^2 + \si_1^2.
\end{cases}
\end{equation}

By analogy with the previous example, let $\si_1^2 = \si_2^2 = \si_4^2 = \si_5^2 = 1$, and $\si_3^2 = 0.9$. Then there are triple collisions among the particles $Y_2$, $Y_3$ and $Y_4$ with positive probability, but a.s. no total collisions. 
\label{N5}
\end{exm}

\begin{exm} {\it An application of Theorem \ref{cams}.} Take $\si_1^2 = \si_3^2 = 10$ and $\si_2^2 = \si_4^2 = 1$. Then by Theorem~\ref{cams} there are a.s. no total collisions, but this fails to satisfy the conditions of Theorem~\ref{totalcor}. This, together with Example~\ref{N4}, shows that none of the two results: Theorem~\ref{totalcor} applied to the case of $N = 4$ particles, and Theorem~\ref{cams}, is stronger than the other one. 
\label{1010}
\end{exm}

\subsection{A sufficient condition for avoiding multicollisions of a given pattern}

For every nonempty finite subset $I \subseteq \MZ$, denote by $\ol{I} := I\cup\{\max I+1\}$ the augmentation of $I$ by the integer following its maximal element. For example, if $I = \{1, 2, 4, 6\}$, then $\ol{I} = \{1, 2, 4, 6, 7\}$. A nonempty finite subset $I \subseteq \MZ$ is called a {\it discrete interval} if it has the form $\{k, k+1, \ldots, l-1, l\}$ for some $k, l \in \MZ,\ k \le l$. For example, the sets $\{2\}, \{3, 4\}, \{-2, -1, 0\}$ are discrete intervals, and the set $\{3, 4, 6\}$ is not. Two disjoint discrete intervals are called {\it adjacent} if their union is also a discrete interval. For example, discrete intervals $\{1, 2\}$ and $\{3, 4\}$ are adjacent, while $\{3, 4, 5\}$ and $\{10, 11\}$ are not. 

Every nonempty finite subset $I \subseteq \MZ$ can be decomposed into a finite union of disjoint non-adjacent discrete intervals: for example, $I = \{1, 2, 4, 8, 9, 10, 11, 13\}$ can be decomposed as $\{1, 2\}\cup\{4\}\cup\{8, 9, 10, 11\}\cup\{13\}$. This decomposition is unique. 
The non-adjacency is necessary for uniqueness: for example, $\{1, 2\}\cup\{4\}\cup\{8, 9, 10\}\cup\{11\}\cup\{13\}$ is also a decomposition into a finite union of disjoint discrete intervals, but $\{8, 9, 10\}$ and $\{11\}$ are adjacent. 

For a vector $\al = (\al_1, \ldots, \al_M)' \in \BR^M$, define
\begin{equation}
\label{CT}
\CT(\al) = \frac{2(M-1)}{M}\SL_{p=1}^M\al_p^2.
\end{equation} 

For every discrete interval $I = \{k, \ldots, l\} \subseteq \{1, \ldots, N\}$, let $\CP(I) := \CP\left(\si_k, \ldots, \si_l\right)$ and $\CT(I) := \CT\left(\si_k, \ldots, \si_l\right)$. 

Consider a subset $I \subseteq \{1, \ldots, N-1\}$. Suppose it has the following decomposition into the union of non-adjacent discrete disjoint intervals:
\begin{equation}
\label{decomp}
I = I_1\cup I_2\cup \ldots \cup I_r.
\end{equation}

\begin{defn} We say that $I$ {\it satisfies assumption (A)} if 
\begin{equation}
\label{conditionA}
\SL_{\substack{j=1\\ j \ne i}}^r\CT(\ol{I}_j) + \CP(\ol{I}_i) \ge 0,\ \ i = 1, \ldots, r.
\end{equation}
We say that $I$ {\it satisfies assumption (B)} if at least one of the following is true:
\begin{itemize}
\item at least two of discrete intervals $I_1, \ldots, I_r$ are singletons;
\item at least one of discrete intervals $I_1, \ldots, I_r$ consists of two elements $\{k-1, k\}$, and the sequence $(\si^2_j)$ has {\it local concavity} at $k$:
\begin{equation}
\label{localconcavity}
\si_k^2 \ge \frac12\left(\si_{k-1}^2 + \si_{k+1}^2\right);
\end{equation}
\item there exists a subset 
$$
I' = I_{i_1}\cup I_{i_2}\cup \ldots \cup I_{i_s}
$$
which satisfies the assumption (A).
\end{itemize}
\label{defnasmp} 
\end{defn}

\begin{rmk} (i) If a subset $I \subseteq \{1, \ldots, N-1\}$ is a discrete interval, that is, the decomposition~\eqref{decomp} is trivial, then Assumption (A) is equivalent to $\CP(\ol{I}) \ge 0$. 

(ii) If a subset $I \subseteq \{1, \ldots, N-1\}$ is a discrete interval of three or more elements, then Assumption (B) is equivalent to $\CP(\ol{I}) \ge 0$.

(iii) If a subset $I \subseteq \{1, \ldots, N-1\}$ contains two elements: $I = \{k, l\},\ k < l$, then Assumption (B) is automatically satisfied if $k + 1 < l$. If $k+1 = l$, then Assumption (B) is equivalent to the local concavity at $l$: 
$$
\si_{l}^2 \ge \frac12\left(\si_{l+1}^2 + \si_{l-1}^2\right).
$$
Indeed, as mentioned in Example~\ref{N3}, the condition $\CP(\ol{I}) \ge 0$ is more restrictive than local concavity at $l$. 
\label{reduct}
\end{rmk}

\begin{thm} Consider a system of competing Brownian particles from Definition~\ref{classicdef}. Fix a subset $J \subseteq \{1, \ldots, N-1\}$. Suppose every subset $I$ such that $J \subseteq I \subseteq \{1, \ldots, N-1\}$ satisfies assumption (B). Then there a.s. does not exist $t > 0$ such that the system has a multicollision with pattern $J$ at time $t$. 
\label{mainthm}
\end{thm}

The following immediate corollary gives a sufficient condition for absence of multicollisions of a given order (and, in particular, multiple collisions of a given order). 

\begin{cor} Consider a classical system of competing Brownian particles from Definition~\ref{classicdef}. Fix an integer $M = 3, \ldots, N$, and suppose that every subset $I \subseteq \{1, \ldots, N-1\}$ with $|I| \ge M$ satisfies condition~\eqref{conditionA}. Then a.s. there does not exist $t > 0$ such that the system has a multicollision (and, in particular, a collision) of order $M$ 
\end{cor}

\subsection{Examples of avoiding multicollisions} In this subsection, we apply Theorem~\ref{mainthm} to systems with a small number of particles: $N = 4$ and $N = 5$. We consider different patterns of multicollisions. 

\begin{exm} {\it Let $N = 4$ (four particles) and $J = \{1, 3\}$.} (This was already mentioned in the Introduction, in Theorem~\ref{example1}.) A multicollision with pattern $J$ is the same as a simultaneous collision of the following type:
\begin{equation}
\label{46}
Y_1(t) = Y_2(t)\ \ \mbox{and}\ \ Y_3(t) = Y_4(t).
\end{equation}
We need to check Assumption (B) for subsets $I = J = \{1, 3\}$ and $I = \{1, 2, 3\}$. The subset $I = \{1, 2, 3\}$ is a discrete interval. According to Remark~\ref{reduct}, we can apply Example~\ref{N4}, and rewrite Assumption (B) as the system of five inequalities~\eqref{CP4I}. For $I = \{1, 3\}$, the decomposition~\eqref{decomp} of $I$ into the union of disjoint non-adjacent discrete intervals has the following form: $I = \{1\}\cup\{3\}$.  Therefore, Assumption (B) is always satisfied. 
Therefore, the system of five inequalities~\eqref{CP4I} is sufficient not only for avoiding total collisions in a system of four particles, but also for avoiding multicollisions~\eqref{46}, with pattern $J = \{1, 3\}$.

\label{weird}
\end{exm}

\begin{exm} {\it Let $N = 4$ and $J  =\{1, 2\}$.} Let us find a sufficient condition for a.s. avoiding triple collisions of the type $Y_1(t) = Y_2(t) = Y_3(t)$. (This was already mentioned in the Introduction, as Theorem~\ref{example3}.) There are two subsets $I$ such that $J \subseteq I \subseteq \{1, 2, 3\}$: $I = \{1, 2\}$ and $I = \{1, 2, 3\}$. These two sets are both discrete intervals. As mentioned in the Remark~\ref{reduct}, Assumption (B) for $I = \{1, 2, 3\}$ is equivalent to 
$\CP(\ol{I}) \ge 0$, which, in turn, is equivalent to~\eqref{CP4I}. 
Assumption (B) for $I = \{1, 2\}$ is equivalent to local concavity at index $2$: $2\si_2^2  \ge \si_1^2 + \si_3^2$. 
 We can write this as the system of six inequalities: local concavity at $2$  and the five inequalities~\eqref{CP4I} from Example~\ref{N4}. 
\label{eleg2} 
\end{exm}

\begin{exm} {\it Consider $N = 5$ (five particles) and take the pattern $J = \{1, 2, 3\}$. } This corresponds to a collision of the following type:
\begin{equation}
\label{0091}
Y_1(t) = Y_2(t) = Y_3(t) = Y_4(t).
\end{equation}
There are two subsets $I$ such that $J \subseteq I \subseteq \{1, 2, 3, 4\}$: $I = J = \{1, 2, 3\}$ and $I = \{1, 2, 3, 4\}$. These two sets are both discrete intervals. As mentioned in the Remark~\ref{reduct}, Assumption (B) for each of these sets $I$ takes the form $\CP(\ol{I}) \ge 0$: $\CP(\{1, 2, 3, 4\}) \ge 0$ and $\CP(\{1, 2, 3, 4, 5\}) \ge 0$. We can write them as the system of twelve inequalities: the five inequalities~\eqref{CP4I} from Example~\ref{N4},  and the seven inequalities~\eqref{CP5I} from Example~\ref{N5}. 
\label{exmfive1}
\end{exm}

\begin{exm} {\it Consider $N = 5$ and take the pattern $J = \{1, 2, 4\}$}. This corresponds to a collision 
\begin{equation}
\label{0092}
Y_1(t) = Y_2(t) = Y_3(t),\ \ \mbox{and}\ \ Y_4(t) = Y_5(t).
\end{equation}
There are two subsets $I$ such that $J \subseteq I \subseteq \{1, 2, 3, 4\}$: $I = J = \{1, 2, 4\}$ and $I = \{1, 2, 3, 4\}$. The set $I = \{1, 2, 3, 4\}$ is a discrete interval; by Remark~\ref{reduct}, Assumption (B) for $I = \{1, 2, 3, 4\}$ takes the form $\CP(\{1, 2, 3, 4, 5\}) \ge 0$. This is equivalent to the conjunction of the seven inequalities~\eqref{CP5I} from Example~\ref{N5}. For $I = \{1, 2, 4\}$, the situation is more complicated. The decomposition of this $I$ into a union of disjoint non-adjacent discrete intervals is $I = \{1, 2\}\cup \{4\}$. Therefore, Assumption (B) holds for this set $I$ in one of the following cases:
\begin{itemize}
\item if there is local concavity at $2$: $\si_2^2 \ge \left(\si_1^2 + \si_3^2\right)/2$;
\item Assumption (A) holds for $\{1, 2\}$, which is equivalent to $\CP(\{1, 2, 3\}) \ge 0$, which, in turn, is a stronger assumption than local concavity at $2$ (see Example~\ref{N3});
\item Assumption (A) holds for $\{4\}$, which is when $\CP(\{4, 5\}) \ge 0$; but this is never true, see~\eqref{CP2};
\item Assumption (A) holds for $\{1, 2\}\cup\{4\}$, which is equivalent to 
\begin{equation}
\label{35403}
\CT(\{1, 2, 3\}) + \CP(\{4, 5\}) \ge 0,\ \ \CT(\{4, 5\}) + \CP(\{1, 2, 3\}) \ge 0.
\end{equation}
\end{itemize}
But $\CP(\{4, 5\}) = \CP(\si_4, \si_5) = -\si_4^2 - \si_5^2$, as in~\eqref{CP2}, and $\CP(\{1, 2, 3\}) = \CP(\si_1, \si_2, \si_3)$ is given by~\eqref{CP3}. Therefore, we have:
\begin{equation}
\label{45}
\CT(\{1, 2, 3\}) + \CP(\{4, 5\}) = \frac43\left(\si_1^2 + \si_2^2 + \si_3^2\right) - \si_4^2 - \si_5^2 \ge 0,
\end{equation}
which can be written as
\begin{equation}
\label{4509}
4\si_1^2 + 4\si_2^2 + 4\si_3^2 \ge 3\si_4^2 + 3\si_5^2.
\end{equation}
The other condition $\CT(\{4, 5\}) + \CP(\{1, 2, 3\}) \ge 0$ is equivalent to the system of the following three inequalities:
\begin{equation}
\label{413}
\begin{cases}
4\si_1^2 + 4\si_3^2  \le 8\si_2^2 + 3\si_4^2 + 3\si_5^2;\\
4\si_1^2  \le 2\si_2^2 + 2\si_3^2 + 3\si_4^2 + 3\si_5^2;\\
4\si_3^2  \le 2\si_1^2 + 2\si_2^2 + 3\si_4^2 + 3\si_5^2.
\end{cases}
\end{equation}
Therefore,~\eqref{35403} is equivalent to the system of~\eqref{4509} and~\eqref{413}:
\begin{equation}
\label{34234}
\begin{cases}
4\si_1^2 + 4\si_3^2  \le 8\si_2^2 + 3\si_4^2 + 3\si_5^2;\\
4\si_1^2  \le 2\si_2^2 + 2\si_3^2 + 3\si_4^2 + 3\si_5^2;\\
4\si_3^2  \le 2\si_1^2 + 2\si_2^2 + 3\si_4^2 + 3\si_5^2;\\
4\si_1^2 + 4\si_2^2 + 4\si_3^2 \ge 3\si_4^2 + 3\si_5^2.
\end{cases}
\end{equation}
 Assumption (B) holds for $I = \{1, 2\}\cup\{4\}$ if and only if there is local concavity at $2$ or~\eqref{34234} hold. Thus, the system of seven inequalities~\eqref{CP5I} from Example~\ref{N5}, together with local concavity at $2$ or the four inequalities~\eqref{34234}, is a sufficient condition for avoiding multicollisions of pattern $\{1, 2, 4\}$. 
\label{exmfive2}
\end{exm}

\begin{rmk} We can also make use of the condition~\eqref{Bruggeman} instead of the five inequalities~\eqref{CP4I}. If the condition~\eqref{Bruggeman} is satisfied, then there are a.s. no simultaneous collisions~\eqref{46} at any time $t > 0$. Similarly, in all of the examples involving $N = 4$ particles avoiding certain types of collisions, we can substitute the condition~\eqref{Bruggeman} instead of the five inequalities~\eqref{CP4I}, 
and the statement will still be true. In Example~\ref{eleg2}, the two conditions:~\eqref{Bruggeman} and the local concavity at the index $2$, guarantee absence of triple collisions $Y_1(t) = Y_2(t) = Y_3(t)$. 
The same works for Examples~\ref{exmfive1} and~\ref{exmfive2}. 
\end{rmk}

\begin{exm} {\it Suppose we have three or more particles: $N \ge 3$.} Consider the case when all diffusion coefficients are equal to one: $\si_1 = \ldots = \si_N = 1$. Then there are no triple and multiple collisions, as well as no multicollisions of order $M \ge 3$. To show this, we do not even need to use Theorem~\ref{mainthm}. Indeed, using Girsanov transformation as in~\cite[Subsection 3.2]{MyOwn3}, we can transform the classical system of competing Brownian particles into $N$ independent Brownian motions with zero drifts and unit diffusions. Since the Bessel process of dimension two a.s. does not return to the origin, there are a.s. no triple collisions and multicollisions of order $M \ge 3$ for the system of independent Brownian motions. 

Still, we can apply our results of this article to the case of unit diffusion coefficients. Consider total collisions and apply Theorem~\ref{totalcor}. Let $\si_1 = \ldots = \si_N = 1$, so that $\si = \mathbf{1} = (1, 1, \ldots, 1)'$; then it is straightforward to calculate that
$$
c_l(\si) = c_l(\si^{\leftarrow}) = 2N - 6,\ l = 1, \ldots, N- 1.
$$
Therefore, we have:
$$
\CP(\si) = \min(c_1(\si),\, \ldots,\, c_{N-2}(\si),\, c_{N-1}(\si),\, c_1(\si^{\leftarrow}),\, \ldots,\, c_{N-2}(\si^{\leftarrow})) = 2N - 6 \ge 0.
$$
Apply Theorem~\ref{totalcor}: the system avoids total collisions. How does this result change if we move the diffusion coefficients $\si_1^2, \ldots, \si_N^2$ a little away from $1$? In other words, if the vector $\si$ is in a small neighborhood of $\mathbf{1} = (1, \ldots, 1)' \in \BR^N$, what can we say about absence of total collisions? 

If $N = 3$, then $\CP(\mathbf{1}) = 0$. Even in a small neighborhood of $\mathbf{1}$, we can have either $\CP(\si) \ge 0$ or $\CP(\si) < 0$. Therefore, we cannot claim that in a certain neighborhood of $\mathbf{1}$ we do not have any total (in this case, triple) collisions. 
This is consistent with the results of \cite{MyOwn3}. Indeed, the inequality~\eqref{concave} takes the form
\begin{equation}
\label{concave3}
\si_2^2 \ge \frac12\left(\si_1^2 + \si_3^2\right).
\end{equation}
This becomes an equality for $\si = (\si_1, \si_2, \si_3)' = \mathbf{1}$.
The point $\mathbf{1}$ lies at the boundary of the set of points in $\BR^3$ given by~\eqref{concave3}. Or, equivalently, in any neighborhood of $\mathbf{1}$ there are both points $\si$ which satisfy~\eqref{concave3} and which do not satisfy~\eqref{concave3}. 

But for $N \ge 4$ (four or more particles), we have: $\CP(\mathbf{1}) > 0$. Since $\CP(\si)$ is a continuous function of $\si$, there exists a neighborhood $\mathcal U$ of $\mathbf{1}$ such that for all $\si \in \mathcal U$ we have: $\CP(\si) > 0$, and the system of competing Brownian particles does not have total collisions. 
\label{girsanov}
\end{exm}

\section{Semimartingale Reflected Brownian Motion (SRBM) in the Orthant}

\subsection{Definitions} We informally described a semimartingale reflected Brownian motion (SRBM) in the orthant in Section 2. Now, let us define this process formally. Fix $d \ge 1$, the dimension. Let us describe the parameters of an SRBM: a {\it drift vector} $\mu \in \BR^d$, a $d\times d$-matrix $R = (r_{ij})_{1 \le i, j \le d}$ with $r_{ii} = 1$, $i = 1, \ldots, d$, which is called a {\it reflection matrix}, and another $d\times d$ symmetric positive definite matrix $A = (a_{ij})_{1 \le i, j \le d}$, which is called a {\it covariance matrix}. Recall the notation: $S = \BR^d_+$. Our goal is to define a Markov process in $S$ which behaves as a Brownian motion with drift vector $\mu$ and covariance matrix $A$ in the interior of  $S$, and reflects in the direction of the $i$th column of $R$ on the face $S_i$, for each $i = 1, \ldots, d$. 

Let $(\Oa, \CF, (\CF_t)_{t \ge 0}, \MP)$ be the standard setting: a filtered probability space with the filtration satisfying the usual conditions. 

\begin{defn} Fix $x \in S$. Consider the following three processes:

(i) a continuous adapted $S$-valued process $Z = (Z(t), t \ge 0)$;

(ii) an $((\CF_t)_{t \ge 0}, \MP)$-Brownian motion $B = (B(t), t \ge 0)$ in $\BR^d$ with drift vector $\mu$ and covariance matrix $A$, starting from $B(0) = x$;

(iii) another continuous adapted $\BR^d$-valued process 
$$
L = (L(t), t \ge 0),\ \ L(t) = (L_1(t), \ldots, L_N(t))',
$$
such that the following is true:
$$
Z(t) = B(t) + RL(t)\ \ \mbox{for}\ \ t \ge 0,
$$
and for each $k = 1, \ldots, d$, the process $L_k$ has the following properties: it is nondecreasing, $L_k(0) = 0$, and 
$$
\int_0^{\infty}Z_k(t)\md L_k(t) = 0,
$$
that is, $L_k$ can increase only when $Z_k = 0$. 

Then the process $Z$ is called a {\it semimartingale reflected Brownian motion} (SRBM) {\it in the orthant} $S$, with {\it drift vector} $\mu$, {\it covariance matrix} $A$, and {\it reflection matrix} $R$, {\it starting from} $x$. The process $B$ is called the {\it driving Brownian motion} for the SRBM $Z$. The process $L$ is called the {\it vector of regulating processes}  for $Z$, and its $i$th component $L_i$ is called the {\it regulating process corresponding to the face} $S_i$ of the boundary $\pa S$. As mentioned before, the process $Z$ is denoted by $\SRBM^d(R, \mu, A)$. 
\end{defn}

Let us define a few classes of matrices; see also a useful equivalent characterization in \cite[Lemma 2.5]{MyOwn3}. 

\begin{defn} Consider a $d\times d$-matrix $R = (r_{ij})_{1 \le i, j \le d}$. It is called a {\it reflection matrix} if $r_{ii} = 1$ for $i = 1, \ldots, d$. It is called a {\it completely-$\CS$ matrix} if for all nonempty $I \subseteq \{1, \ldots, d\}$ there exists $u \in \BR^{|I|}$ such that $u > 0$ and $[R]_Iu > 0$. It is called a {\it reflection nonsingular $\CM$-matrix} if $r_{ii} = 1$ for $i = 1, \ldots, d$,  $r_{ij} \le 0$ for $i \ne j$, and the spectral radius of the matrix $I_d - R$ is less than $1$. It is called {\it strictly copositive} if it is symmetric and for every $x \in S\setminus\{0\}$ we have: $x'Rx > 0$. 
\end{defn}

Now, let us state a general existence and uniqueness theorem for this process was proved in \cite{RW1988, TW1993, HR1981a}; see also the survey \cite{Wil1995}. 

\begin{prop} Fix any point $x \in S$. If $R$ is a reflection nonsingular $\CM$-matrix, then there exists and is unique in the strong sense an $\SRBM^d(R, \mu, A)$, starting from $x$. If, more generally, $R$ is a completely-$\CS$ reflection matrix, then this process exists and is unique in the weak sense. These processes for all $x \in S$ together form a Feller continuous strong Markov family. 
\label{existence}
\end{prop}

An SRBM in the orthant is the {\it heavy traffic limit} for series of queues, when the traffic intensity tends to one, see \cite{ReimanThesis, Rei1984, Har1978}. We can also define an SRBM in general convex polyhedral domains in $\BR^d$, see \cite{DW1995}. An SRBM in the orthant and in convex polyhedra has been extensively studied, see the survey \cite{Wil1995}, and articles  \cite{HR1981a, HR1981b, HW1987a, HW1987b, Wil1987,  RW1988, TW1993, DW1994, DK2003, Chen1996, Har1978, BDH2010, BL2007, DW1995, DH2012, DH1992, HH2009, Wil1985a, Wil1985b, Wil1998b, K1997, K2000, KW1996, R2000, KR2012a, KR2012b, KW1992a}.

For any subset $I \subseteq \{1, \ldots, d\}$, we let $S_I := \cap_{i \in I}S_i = \{x \in S\mid x_i = 0\ \mbox{for}\ i \in I\}$. If an $\SRBM^d(R, \mu, A)$ starts from $z \in S$, we denote the corresponding probability measure as $\MP_z$. 

The following proposition was shown in \cite[Subsection 3.2]{MyOwn3}.

\begin{prop} Fix a $d\times d$ reflection nonsingular $\CM$-matrix $R$ and a positive definite symmetric $d\times d$-matrix $A$. 
Let $\mu \in \BR^d$. Denote $Z = (Z(t), t \ge 0) = \SRBM^d(R, \mu, A)$. Consider the statement
\begin{equation}
\label{Indepe}
\MP_z\left(\exists\ t > 0:\ Z(t) \in S_I\right) = 0,
\end{equation}
Whether it is true or false does not depend on the initial condition $z \in S$ or on the drift vector $\mu \in \BR^d$; it depends only on the reflection matrix $R$ and the covariance matrix $A$.
\label{432}
\end{prop}

This justifies the following definition, taken from \cite{MyOwn3}. 

\begin{defn} An $\SRBM^d(R, \mu, A)$ {\it does not hit} $S_I$, or {\it avoids} $S_I$, if for every $z \in S$, we have: 
\begin{equation}
\label{pa}
\MP_z\left(\exists\ t > 0: Z(t) \in S_I\right) = 0.
\end{equation}
An $\SRBM^d(R, \mu, A)$ {\it hits} $S_I$ if for every $z \in S$, we have: 
\begin{equation}
\label{nnpa}
\MP_z\left(\exists\ t > 0: Z(t) \in S_I\right) > 0.
\end{equation}
\end{defn}

\begin{rmk} For every fixed nonempty $I \subseteq \{1, \ldots, d\}$, either~\eqref{pa} or~\eqref{nnpa} holds. In other words, either all the probabilities on the left-hand side of~\eqref{pa} are simultaneously positive for all $z \in S$, or they are all simultaneously equal to zero for all $z \in S$. The property of a.s. avoiding a given edge does not depend on the initial condition, or on the drift vector $\mu$. However, if the probability on the left-hand side~\eqref{pa} is positive, then its exact value does depend on $z$ and $\mu$; see Remark 5 in \cite[Subsection 3.2]{MyOwn3}.
\end{rmk}
%

\begin{defn} An $\SRBM^d(R, \mu, A)$ {\it avoids edges of order $p$},  if it does not hit $S_I$ for any subset $I \subseteq \{1, \ldots, d\}$ with  $p$ elements.  An $\SRBM^d(R, \mu, A)$ {\it avoids the corner} if it does not hit edges of order $p = d$. An $\SRBM^d(R, \mu, A)$ {\it avoids non-smooth parts of the boundary} if it does not hit edges of order $p = 2$. 
\end{defn}

\subsection{Main Results}

Here, we state our main results for an SRBM in the orthant. There are three important theorems. First, we provide a sufficient condition for not hitting the corner, and another sufficient condition for hitting the corner. Taken together, they do not give us a necessary and sufficient condition, because there is a gap between them. In this respect, the results of this paper is different from that of \cite{MyOwn3}, where we gave a necessary and sufficient condition for avoiding non-smooth parts of the boundary. 

A remaining question is about hitting or avoiding a given edge $S_I$ of the boundary $\pa S$. We provide another theorem which reduces it to  the question of not hitting the corner. This gives us a sufficient condition for not hitting the given edge of $\pa S$. 
The last of these three main results is a sufficient condition for hitting a given edge of 
$\pa S$.

For a strictly copositive $d\times d$-matrix $Q = (q_{ij})_{1 \le i, j \le d}$ and consider the following constants:
$$
c_+(Q) := \max\limits_{x \in S\setminus\{0\}}\frac{x'QAQx}{x'Qx},\qquad
c_-(Q) := \min\limits_{x \in S\setminus\{0\}}\frac{x'QAQx}{x'Qx}.
$$

\begin{lemma} For a positive definite matrix $A$, the numbers $c_{\pm}(Q)$ are well defined and 
strictly positive. 
\label{constants}
\end{lemma}

The (rather straightforward) proof is postponed until the Appendix. 
The following theorem is our main result about an SRBM hitting the corner. 

\begin{thm}
\label{cornerthm}
Suppose $R$ is a completely-$\CS$ reflection matrix, and $A$ is a positive definite symmetric matrix. Take a strictly copositive nonsingular $d\times d$-matrix $Q$.

(i) If the following conditions are true:
\begin{equation}
\label{Ncorner}
\tr\left(QA\right) \ge 2c_+(Q),\ \ \mbox{and}\ \ (QR)_{ij} \ge 0\ \ \mbox{for}\ \ i \ne j,
\end{equation}
then the $\SRBM^d(R, \mu, A)$ does not hit the corner.

(ii) If the following conditions are true:
\begin{equation}
\label{Ycorner}
0 \le \tr\left(QA\right) < 2c_-(Q),\ \ \mbox{and}\ \ (QR)_{ij} \le 0\ \ \mbox{for}\ \ i \ne j,
\end{equation}
then the $\SRBM^d(R, \mu, A)$ hits the corner. 
\end{thm}

An important example of such matrix $M$ is given in the next corollary.

\begin{cor} Suppose the matrix $R$ is a reflection nonsingular $\CM$-matrix for which there exists a diagonal matrix $C = \diag(c_1, \ldots, c_d)$ with $c_1, \ldots, c_d > 0$, such that $\ol{R} = RC$ is symmetric.  

(i) If the following condition is true:
\begin{equation}
\label{Ncorner-om}
\tr\left(\om A\right) \ge 2c_+(\om),
\end{equation}
then the $\SRBM^d(R, \mu, A)$ does not hit the corner.

(ii) If the following conditions are true:
\begin{equation}
\label{Ycorner-om}
0 \le \tr\left(\om A\right) < 2c_-(\om),
\end{equation}
then the $\SRBM^d(R, \mu, A)$ hits the corner. 
\label{cor:inv-R}
\end{cor}

Theorem~\ref{cornerthm} applies also to completely-$\CS$ reflection matrices which are not reflection nonsingular $\CM$-matrices (that is, they contain positive off-diagonal elements). The following is a useful corollary, which can be viewed as a generalization of 
Corollary~\ref{cor:inv-R}, because for a reflection nonsingular $\CM$-matrix $R$ we have: $\tilde{R} = R$. 

\begin{cor}
Take a completely-$\CS$ reflection matrix $R = (r_{ij})_{1 \le i, j \le d}$. Consider the matrix $\tilde{R} = (\tilde{r}_{ij})_{1 \le i, j \le d}$, defined as
$$
\tilde{r}_{ij} = 
\begin{cases}
1,\ i = j;\\
r_{ij},\ r_{ij} \le 0;\\
0,\ r_{ij} > 0.
\end{cases}
$$
Assume $\tilde{R}$ is a reflection nonsingular-$\CM$ matrix. Also, suppose there exists a diagonal matrix $C = \diag(c_1, \ldots, c_d)$ with $c_1, \ldots, c_d > 0$, such that $\ol{R} = \tilde{R}C$ is symmetric.  If the following condition is true:
\begin{equation}
\label{Ncorner-om-notM}
\tr\left(\om A\right) \ge 2c_+(\om),
\end{equation}
then the $\SRBM^d(R, \mu, A)$ does not hit the corner.
\end{cor}

\begin{proof} Just take $Q = \tilde{R}^{-1}$ and apply Theorem~\ref{cornerthm}.
We need only to prove that $(QR)_{ij} \ge 0$ for $i \ne j$. But all elements of $Q$ are nonnegative, and so 
$$
(QR)_{ij} = \SL_{k=1}^dq_{ik}r_{kj} \ge \SL_{k=1}^dq_{ik}\tilde{r}_{kj} = (Q\tilde{R})_{ij} = 0.
$$
\end{proof}

\begin{exm} Let $d = 2$, and
$$
R = 
\begin{bmatrix}
1 & 2\\
3 & 1
\end{bmatrix},
\ \ 
A = I_2 = 
\begin{bmatrix}
1 & 0\\
0 & 1
\end{bmatrix}
$$
Then $\tilde{R} = I_2$ and $\om = I_2$. Therefore, 
$$
c_+(\om) =  \max\limits_{x \in S\setminus\{0\}}\frac{x'I_2x}{x'I_2x} = 1,
\ \ \mbox{and}\ \ \tr\left(\om A\right) = \tr(I_2) = 2.
$$
Thus, condition~\eqref{Ncorner-om-notM} holds, and the $\SRBM^2(R, \mu, A)$ does not hit the corner.
\end{exm}

Sometimes the numbers $c_{\pm}(Q)$ are difficult to calculate. Let us give useful estimates of $c_+(Q)$ from above, and of $c_-(Q)$ from below. 

\begin{lemma} If the matrix $Q = (q_{ij})_{1 \le i, j \le d}$ satisfies $q_{ij} > 0$ for all $i, j = 1, \ldots, d$, then 
$$
c_+(Q) \le \ol{c}_+(Q) := \max\limits_{1 \le i \le j \le d}\frac{(QAQ)_{ij}}{q_{ij}},
\ \ 
c_-(Q) \ge \ol{c}_-(Q) := \min\limits_{1 \le i \le j \le d}\frac{(QAQ)_{ij}}{q_{ij}}.
$$
\label{simple}
\end{lemma}

The next theorem establishes a connection between not hitting the corner and not hitting an edge. It is similar to results from \cite{IKS2013}, and we took the proof technique from \cite{IKS2013}. 

\begin{thm}
\label{corner2edge} Take a completely-$\CS$ reflection matrix $R$. Consider an $\SRBM^d(R, \mu, A)$. Fix a nonempty subset $J \subseteq \{1, \ldots, d\}$. Suppose   the process 
$$
\SRBM^{|I|}([R]_I, [\mu]_I, [A]_I)\ \ \mbox{for each}\ \ J \subseteq I \subseteq \{1, \ldots, d\}
$$
does not hit the corner. Then an $\SRBM^d(R, \mu, A)$ does not hit the edge $S_J$. 
\end{thm}

The last of our main results about SRBM links hitting corners to hitting edges. Note that in this case, the condition that $R$ is a reflection nonsingular $\CM$-matrix, rather than just a completely-$\CS$ matrix, is crucial; the reader can see that the proof heavily uses comparison techniques from \cite{MyOwn2}, which, in turn, apply the condition that $R$ is a reflection nonsingular $\CM$-matrix. 

\begin{thm} 
\label{corner2edge2}
Consider an $\SRBM^d(R, \mu, A)$ with a reflection nonsingular $\CM$-matrix $R$. Fix a nonempty subset $I \subseteq \{1, \ldots, d\}$. Suppose an $\SRBM^{|I|}([R]_I, [\mu]_I, [A]_I)$ hits the corner. Then an $\SRBM^d(R, \mu, A)$ hits the edge $S_I$. 
\end{thm}

The rest of the section will be devoted to the proofs of Theorems~\ref{cornerthm}, ~\ref{corner2edge} and~\ref{corner2edge2}. 

\subsection{Proof of Theorem~\ref{cornerthm}} First, we present an informal overview of the proof, and then   give a complete proof. 

\subsubsection{Outline of the proof.} First, we show (i). Let $Z = (Z(t), t \ge 0)$ be an $\SRBM^d(R, \mu, A)$, starting from $z \in S$. By Proposition~\ref{432}, we can assume $z \in S\setminus\pa S$, and  $\mu = 0$. Consider the function
\begin{equation}
\label{101}
F(x) := x'Qx.
\end{equation}
The matrix $Q$ is strictly copositive. Therefore, if $F(x) = 0$ for a certain $x \in S$, then $x = 0$. Therefore, the process $Z$ hits the corner if and only if the process $F(Z(\cdot))$ hits zero. Let $L = (L(t), t \ge 0)$ be the vector of regulating processes for $Z$, and let $B = (B(t), t \ge 0)$ be the driving Brownian motion for $Z$, so that we have:
\begin{equation}
\label{6112}
Z(t) = B(t) + RL(t),\ t \ge 0.
\end{equation}
If we write an equation for $F(Z(\cdot))$ using the It\^o-Tanaka formula, we get:
$$
\md F(Z(t)) = \be(t)\md t + \md M(t) + \md l(t),
$$
where $l = (l(t), t \ge 0)$ is a {\it nondecreasing} process, and $M = (M(t), t \ge 0)$ is a local martingale. Since we wish to prove that $F(Z(t)) > 0$ for all $t > 0$, we can eliminate the term $l$. In other words, it suffices to prove this property of staying positive for the process $U = (U(t), t \ge 0)$ given by
$$
\md U(t) = \be(t)\md t + \md M(t).
$$
It turns out that the drift coefficient $\be$ is constant. The local martingale $M = (M(t), t \ge 0)$ can be represented as $\md M(t) = \rho(t)\md W(t)$, where $W = (W(t), t \ge 0)$ is a standard Brownian motion. The diffusion coefficient $\rho(t)$ is comparable with that in the SDE for Bessel squared process. After an appropriate random time-change, we can make the diffusion coefficient exactly equal to the one for a Bessel squared process. However, this will not turn our process into a Bessel squared process. Indeed, the drift coefficient for the new process will not be constant (and for a Bessel squared process, it is constant). Still, we can bound this drift coefficient from below by $2$. But we know that a Bessel squared process hits zero if and only if its index is less than two. Therefore, our time-changed process (together with the original process $F(Z(\cdot))$) does not hit zero. This, in turn, means that the process $Z$ does not hit the origin. 
The proof of (ii) is similar, only there the process $l = (l(t), t \ge 0)$ is nonincreasing, and we bound the new drift coefficient from above by something strictly less than $2$. Together, this means that  the process $F(Z(\cdot))$ does indeed hit zero, and the process $Z$ hits the origin.

\subsubsection{Complete proof} 

We split the proof into several lemmata.

\begin{lemma}
The process $F(Z(\cdot))$ can be represented as 
\begin{equation}
\label{eq:rep-F-Z}
\md F(Z(t)) = \rho(t)\md\ol{W}(t) + \tr\bigl(QA\bigr)\md t + l(t),
\end{equation}
where $\ol{W} = (\ol{W}(t), t \ge 0)$ is a standard Brownian motion, 
\begin{equation}
\label{eq:rep-q}
\rho(t) := \left(Z'(t)QAQZ(t)\right)^{1/2},\ \ t \ge 0.
\end{equation}
and $l = (l(t), t \ge 0)$ is a continuous nondecreasing process with $l(0) = 0$. 
\label{lemma:rep-F-Z}
\end{lemma}

The proof of this lemma involves little more than applying It\^o-Tanaka's formula and some computations. It is postponed until the end of this subsection. Assuming we established this lemma, let us complete the proof. 

Define
$$
\tau := \inf\{t \ge 0\mid Z(t) = 0\} = \inf\{t \ge 0\mid F(Z(t)) = 0\}.
$$
Then $\tau > 0$ a.s., because $Z(0) = x > 0$. For $s < \tau$, we have  $Z(s) \in S\setminus\{0\}$, and $F(Z(s)) > 0$. It follows from the definition of constants $c_{\pm}$ that
$$
\frac12c_-^{1/2} \le \frac{\rho(s)}{2F^{1/2}(Z(s))}  = \frac12\left(\frac{Z'(s)QAQZ(s)}{Z'(s)QZ(s)}\right)^{1/2} \le \frac12c_+^{1/2}.
$$
Make the following time change:
$$
\Delta(t) := \int_0^t\frac{\rho^2(s)}{4F(Z(s))}\md s,\ \ t \le \tau.
$$
By \cite[Lemma 2]{MyOwn1}, this is a strictly increasing function on $[0, \tau]$ with $\De(0) = 0$. Denote $\tau_0 := \De(\tau)$. Define the inverse of $\De$ by 
$$
\chi(s) := \inf\{t \ge 0\mid \De(t) \ge s\}.
$$

\begin{lemma} The time-changed process 
$$
V = (V(s), s < \tau_0),\ \ \mbox{defined by}\ \ V(s) \equiv F(Z(\chi(s))), 
$$
satisfies the following equation:
\begin{equation}
\label{eq:process-V}
\md  F(Z(\chi(s))) = \be(s)\md s + 2V^{1/2}(s)\md W(s) + \ol{l}(s),
\end{equation}
where $\ol{l}(s) := l(\chi(s))$ is a nondecreasing process, $\be = (\be(s), s < \tau_0)$ is a certain drift coefficient satisfying $\be(s) \ge \be \ge 2$ for all $s \in [0, \tau_0)$, and $W = (W(s), s \ge 0)$ is a standard Brownian motion. 
\end{lemma}

 \begin{proof}
By \cite[Lemma 2]{MyOwn1}, the process $V = (V(s), s \ge 0)$ satisfies the following equation:
$$
\md V(s) = \tr\bigl(QA\bigr)\frac{V(s)}{\rho^2(\chi(s))}\md s + 2V^{1/2}(s)\md W(s) + l(\chi(s)).
$$
Here, $W = (W(t), t \ge 0)$ is yet another standard Brownian motion.  
Note that 
$$
\frac14c_- \le \De'(s) = \frac{\rho^2(s)}{4F(Z(s))} = \frac{Z'(s)QAQZ(s)}{4Z'(s)QZ(s)} \le \frac14c_+.
$$
Therefore, the mapping $\De : [0, \tau) \to [0, \tau_0)$ is one-to-one, and $\tau = \infty$ if and only if $\tau_0 = \infty$. 
Note that
$$
\frac{V(s)}{\rho^2(\chi(s))} \ge c_+^{-1},
$$
and $\tr(QA) \ge 2c_+ \ge 0$. Therefore,
$$
\tr\bigl(QA\bigr)\frac{V(s)}{\rho^2(\chi(s))} \ge \tr\bigl(QA\bigr)c_+^{-1} =: \be \ge 2.
$$
This completes the proof. 
\end{proof}

Now, note that we have: 
$$
\MP\left(\exists t > 0: F(Z(t)) = 0\right) = 0\ \ \mbox{if and only if}\ \ \MP\left(\exists s > 0: V(s) = 0\right) = 0. 
$$
Suppose the condition~\eqref{Ncorner} holds. We need to prove that the process $Z$ does not hit the corner. Assume the converse. Then $\MP(\tau < \infty) > 0$, and
$\MP(\tau_0 < \infty) > 0$. On the event $\{\tau_0 < \infty\}$, we have: $V(\tau_0) = 0$. Consider the squared Bessel process $\ol{V} = (\ol{V}(s), s \ge 0)$, given by the equation
$$
\md\ol{V}(s) = 2\ol{V}^{1/2}(s)\md W(s) + \be \md s,\ \ \ol{V}(0) = V(0).
$$
Since $\be \ge 2$, it is known (see, e.g., \cite[Section 11.1, p. 442]{RevuzYorBook}) that $\ol{V}$ a.s. does not hit $0$. It follows from Lemma~\ref{lemma:aux-comparison} in the Appendix that $V(s) \ge \ol{V}(s)$ a.s. for $s < \tau_0$. If $\tau_0 < \infty$, then by continuity $V(\tau_0) \ge \ol{V}(\tau_0) > 0$, but $V(\tau_0) = 0$. This contradiction completes the proof of (i). The proof of (ii) is similar.

\medskip
\noindent
{\it Proof of Lemma~\ref{lemma:rep-F-Z}.} Recall the definition of function $F$ from~\eqref{101}. Since the matrix $Q$ is strictly copositive, we have: $F(x) > 0$ for $x \in S\setminus\{0\}$. Since the matrix $Q$ is symmetric, the first and second order derivatives of the function $F$ are
$$
\frac{\pa F}{\pa x_i} = \bigl(2Qx\bigr)_i = 2\SL_{k=1}^dq_{ik}x_k,\ \ \frac{\pa^2F}{\pa x_i\pa x_j} = 2q_{ij},\ \ i, j = 1, \ldots, d.
$$
Note that $\langle Z_i, Z_j\rangle_t = \langle B_i, B_j\rangle_t = a_{ij}t$. 
By the It\^o-Tanaka formula applied to the process $Z$ from~\eqref{6112} and the function $F$ from~\eqref{101}, we have:
\begin{align*}
\md F(Z(t)) =& \SL_{i=1}^d\frac{\pa F}{\pa x_i}(Z(t))\md Z_i(t) + 
\frac12\SL_{i=1}^d\SL_{j=1}^d\frac{\pa^2F}{\pa x_i\pa x_j}(Z(t))\md\langle Z_i, Z_j\rangle_t \\ &=\SL_{i=1}^d\bigl(2QZ(t)\bigr)_i\md B_i(t) + 
\SL_{i=1}^d\SL_{k=1}^d\bigl(2QZ(t)\bigr)_ir_{ik}\md L_k(t)  + \SL_{i=1}^d\SL_{j=1}^dq_{ij}a_{ij}\md t\\ &=
2\SL_{i=1}^d\SL_{j=1}^dq_{ij}Z_j(t)\md B_i(t) + 
2\SL_{i=1}^d\SL_{j=1}^d\SL_{k=1}^dq_{ij}Z_j(t)r_{ik}\md L_k(t) + \tr\bigl(QA\bigr)\md t.
\end{align*}
Let us show that the following process is nondecreasing: 
$$
l(t) := \int_0^t2\SL_{i=1}^d\SL_{j=1}^d\SL_{k=1}^dq_{ij}Z_j(s)r_{ik}\md L_k(s)\md s. 
$$
Indeed,
\begin{align*}
\SL_{i=1}^d\SL_{j=1}^d\SL_{k=1}^d&q_{ij}Z_j(t)r_{ik}\md L_k(t) =
\SL_{i=1}^d\SL_{j=1}^d\SL_{k=1}^d\rho_{ji}Z_j(t)r_{ik}\md L_k(t) \\ & = \SL_{j=1}^d\SL_{k=1}^d\bigl(QR\bigr)_{jk}Z_j(t)\md L_k(t) \\ & = 
 \SL_{j=1}^d\bigl(QR\bigr)_{jj}Z_j(t)\md L_j(t) +  \SL_{k \ne j}\bigl(QR\bigr)_{jk}Z_j(t)\md L_k(t).
\end{align*}
For each $j = 1, \ldots, d$, the regulating process $L_j$ can grow only if $Z_j = 0$: we express this by writing $Z_j(t)\md L_j(t) = 0$. And for $k \ne j$, we have: $(QR)_{jk} \ge 0$ by assumptions of Theorem~\ref{cornerthm}, and $Z_j(t) \ge 0, \md L_k(t) \ge 0$ by definition. Therefore, the process $l$ is nondecreasing. 

Now, recall that $B_1, \ldots, B_d$ are driftless one-dimensional Brownian motions (they are driftless, because the drift $\mu = 0$, according to our assumptions). Therefore, the following process is a continuous local martingale:
$$
M = (M(t), t \ge 0),\ \ M(t) := 2\SL_{i=1}^d\SL_{j=1}^d\int_0^t\rho_{ij}Z_j(s)\md B_i(s).
$$
We can represent the process $F(Z(\cdot))$ as follows:
$$
\md F(Z(t)) = \md M(t) + \tr\bigl(QA\bigr)\md t + \md l(t).
$$
Let us calculate the quadratic variation of $M$. Recall that, by definition of the process $B$, $\langle B_i, B_j\rangle_t = a_{ij}t$. Let
$$
M_{ij}(t) = \int_0^t\int_0^tZ_j(s)q_{ij}\md B_i(s),\ \ i, j = 1, \ldots, d.
$$
For $i, j, k, l = 1, \ldots, d$, we have:
$$
\langle M_{ij}, M_{kl}\rangle_t = \int_0^tZ_j(s)q_{ij}Z_l(s)q_{kl}a_{ik}\md s.
$$
But the quadratic variation of $M = \sum_{i=1}^d\sum_{j=1}^dM_{ij}$ is equal to the sum
\begin{align*}
\langle M\rangle_t = & \SL_{i=1}^d\SL_{j=1}^d\SL_{k=1}^d\SL_{l=1}^d\langle M_{ij}, M_{kl}\rangle_t = 
\SL_{i=1}^d\SL_{j=1}^d\SL_{k=1}^d\SL_{l=1}^d\int_0^t
Z_j(s)q_{ij}Z_l(s)q_{kl}a_{ik}\md s \\ &= 
\SL_{i=1}^d\SL_{j=1}^d\SL_{k=1}^d\SL_{l=1}^d\int_0^t
Z_j(s)q_{ij}a_{ik}q_{kl}Z_l(s)\md s = 
\int_0^t\left(Z'(s)QAQZ(s)\right)\md s.
\end{align*}
Then we can represent $M$ as the stochastic integral 
$$
M(t) = 2\int_0^t\rho(s)\md \ol{W}(s),
$$
where $\ol{W} = (\ol{W}(t), t \ge 0)$ is a standard Brownian motion. This completes the proof of Lemma~\ref{lemma:rep-F-Z}. 

\subsection{Proof of Theorem~\ref{corner2edge}}

We prove this theorem using induction by $d$. The induction base is trivial. Induction step: assume that the statement is true for $d-1$ instead of $d$, and try to prove it for 
$d$. For $\eps \in (0, 1)$, let $K_{\eps} = \{x \in S\mid \eps \le \norm{x} \le \eps^{-1}\}$. Fix a point $z \in S\setminus\{0\}$, so that $z \in K_{\eps}$ for all $\eps > 0$ small enough. Start a copy of an $\SRBM^d(R, \mu, A)$ from $z$ (we can assume this by Proposition~\ref{432}). Denote this copy by $Z = (Z(t), t \ge 0)$, and let $B = (B(t), t \ge 0)$ be its driving Brownian motion. 
Let 
$$
\tau := \inf\{t \ge 0\mid Z(t) \in S_I\}
$$
be the first moment when the process $Z$ hits the edge $S_I$. We need to show that $\tau = \infty$ a.s. Let 
$$
\eta_{\eps} := \inf\{t \ge 0\mid Z(t) \in K_{\eps}\}.
$$
Note that $\eta_{\eps} \le \eta_{\eps'}$ when $\eps' \le \eps$, and $\lim_{\eps \downarrow 0}\eta_{\eps} = \infty$, because by assumptions of the theorem the process $Z$ does not hit the corner: $Z(t) \ne 0$ for all $t \ge 0$ a.s. It suffices to show that $\tau \ge \eta_{\eps}$ for all $\eps \in (0, 1)$. Fix an $\eps \in (0, 1)$. For every $x \in K_{\eps}$, there exists an open neighborhood $U(x)$ of $x$ with the following property: there exists some index $i = i(x) \in \{1, \ldots, d\}$ such that for all $y \in U(x)$ we have: $y_{i(x)} > 0$. Since $K_{\eps}$ is compact, we can extract a finite subcover $U(x_1), \ldots, U(x_s)$. Without loss of generality, let us include the neighborhood $U(x_0)$ of $x_0 = z$ into this subcover. Now, define a sequence of stopping times:
$$
\tau_0 := 0,\ \ j_0 := 0;\ \ 
\tau_{k+1} := \inf\{t \ge \tau_k\mid Z(t) \notin U\left(x_{j_k}\right)\},
$$
and $j_{k+1}$ is defined as any $j = 0, \ldots, s$ such that $Z\left(\tau_{k+1}\right) \in U(x_j)$. Suppose that, at some point, we cannot find such $j$; in other words, 
$$
Z(\tau_{k+1}) \notin U\left(x_{j_0}\right)\cup U\left(x_{j_1}\right)\cup\ldots\cup U\left(x_{j_s}\right).
$$
Then the sequence of stopping times terminates, and we denote $K := k+1$. In this case, we have defined $\tau_0, j_0, \tau_1, j_1, \ldots, \tau_{K-1}, j_{K-1}, \tau_K$. If the sequence does not terminate, we let $K = \infty$. We have:
$$
Z_{j_k}(t) > 0\ \ \mbox{for}\ \ t \in [\tau_k, \tau_{k+1}),\ \ k < K.
$$
The sequence $(\tau_{k})$ can be either finite or countable. Recall that $U(x_j),\ j = 0, \ldots, s$ is a cover of $K_{\eps}$. Therefore, $\sup_{k}\tau_k \ge \eta_{\eps}$. It suffices to show that $\tau \ge \tau_{k}$. We prove this using induction by $k$. 

\medskip

Induction base: $k = 1$. If $j_0 \in I$, then $Z_{j_0}(t) > 0$ for $t < \tau_1$, and $Z(t) \notin S_I$. In this case, $\tau \ge \tau_1$ is straightforward. Now, if $j_0 \notin I$, then consider the set $J := \{1, \ldots, d\}\setminus\{j_0\}$. We have the following representation:
$$
\left([Z(t\wedge\tau_1)]_J,\ t \ge 0\right) = (\ol{Z}(t\wedge\tau_1),\ t \ge 0),
$$
where $\ol{Z} = (\ol{Z}(t), t \ge 0)$ is an $\SRBM^{d-1}([R]_J, [\mu]_J, [A]_J)$, starting from $[z]_J$, with the driving Brownian motion $[B]_J = ([B(t)]_J, t \ge 0)$. This process $\ol{Z}$ is well defined, since the matrix $[R]_J$ is a reflection nonsingular $\CM$-matrix, and by Proposition~\ref{existence} there exists a strong version of $\ol{Z}$. By the induction hypothesis, a.s. there does not exist $t \ge 0$ such that $\ol{Z}(t) \in S_I$. For every $y \in S$, we have: $y \in S_I$ if and only if $[y]_J \in S_I$. Therefore, for all $t < \tau_1$ we have: $Z(t) \notin S_I$. This proves that $\tau \ge \tau_1$. 

\medskip

Induction step: suppose $t \ge \tau_k$ and $k < K$, that is, the sequence does not terminate at this step. Then we need to prove $\tau \ge \tau_{k+1}$. Consider the process $(Z(t + \tau_k), t \ge 0)$. This is a version of an $\SRBM^d(R, \mu, A)$, started from $Z(\tau_k)$. But 
$$
Z(\tau_{k}) \in U\left(x_{j_0}\right)\cup U\left(x_{j_1}\right)\cup\ldots\cup U\left(x_{j_s}\right).
$$
There exists $j = 0, \ldots, s$ such that $Z(\tau_k) \in U(x_j)$. In addition, $Z(\tau_k) \in S\setminus\{0\}$, because by induction hypothesis, the process $Z$ never hits the corner. Apply the reasoning from the induction base to this process instead of the original SRBM. The moment $\tau_{k+1} - \tau_k$ plays the role of $\tau_1$ above, and the moment $\tau - \tau_k$ plays the role of $\tau$. Therefore, $\tau - \tau_k \ge \tau_{k+1} - \tau_k$, and $\tau \ge \tau_{k+1}$. This completes the proof.

\subsection{Proof of Theorem~\ref{corner2edge2}}

This theorem is proved using {\it stochastic comparison}. 

\begin{defn} Consider two $\BR^d$-valued processes $Z = (Z(t), t \ge 0)$ and $\ol{Z} = (\ol{Z}(t), t \ge 0)$. We say that $Z$ is {\it stochastically dominated by} $\ol{Z}$, and write it as $Z \preceq \ol{Z}$, if for every $t \ge 0$ and $y \in \BR^d$ we have: 
$$
\MP(Z(t) \ge y) \le \MP(\ol{Z}(t) \ge y).
$$
\end{defn}

\begin{prop}
Take a $d\times d$ reflection nonsingular $\CM$-matrix $R$, a $d\times d$ positive definite symmetric matrix $A$, and a drift vector $\mu \in \BR^d$. Fix a nonempty subset $I \subseteq \{1, \ldots, d\}$. Let 
$$
Z = \SRBM^d(R, \mu, A),\ \ \ol{Z} = \SRBM^{|I|}([R]_I, [\mu]_I, [A]_I)
$$
such that $[Z(0)]_I$ has the same law as $\ol{Z}(0)$. Then $[Z]_I \preceq \ol{Z}$. 
\label{propcomp}
\end{prop}

This result was shown in \cite[Corollary 3.6]{MyOwn2}; it is an easy corollary of general comparison techniques for reflected processes developed in \cite[Theorem 4.1]{R2000}, see also \cite[Theorem 1.1(i)]{KR2012b}, \cite[Theorem 3.1]{Haddad2010}, \cite[Theorem 6(i)]{KW1996}. Now, it is easy to see that Theorem~\ref{corner2edge2} trivially follows from Proposition~\ref{propcomp}. 

\subsection{Corollaries of the main results for an SRBM}

The following corollary of Theorem~\ref{corner2edge} gives a sufficient condition for not hitting edges of a given order. 

\begin{cor} Consider an  $\SRBM^d(R, \mu, A)$. Fix $p = 2, \ldots, d-1$. Suppose for every $I \subseteq \{1, \ldots, d\}$ such that $|I| \ge p$ the process $\SRBM^{|I|}([R]_I, [\mu]_I, [A]_I)$ does not hit the corner. Then an $\SRBM^d(R, \mu, A)$ does not hit edges of order $p$. 
\end{cor}

The next corollary combines the results of Corollary~\ref{cor:inv-R}, Theorem~\ref{corner2edge} and Theorem~\ref{corner2edge2}. Its proof is trivial and is omitted. 

\begin{cor} Take an $\SRBM^d(R, \mu, A)$. Suppose the matrix $R$ is a reflection nonsingular $\CM$-matrix and there exists a diagonal matrix $C = \diag(c_1, \ldots, c_d)$ with $c_1, \ldots, c_d > 0$ such that $RC = \ol{R}$ is symmetric. 
\label{general}

(i) Fix a nonempty subset $J \subseteq \{1, \ldots, d\}$. Suppose that for every subset $I$ such that $J \subseteq I \subseteq \{1, \ldots, d\}$ we have:
\begin{equation}
\label{Nedge}
\tr\left([\oR]_I^{-1}[A]_I\right) \ge 2\!\!\!\!\max\limits_{x \in \BR^{|I|}_+\setminus\{0\}}\frac{x'[\oR]_I^{-1}[A]_I[\oR]_I^{-1}x}{x'[\oR]_I^{-1}x}.
\end{equation}
Then the $\SRBM^d(R, \mu, A)$ avoids $S_I$.

(ii) Fix $p = 1, \ldots, d-1$. Suppose for every subset $I \subseteq \{1, \ldots, d\}$ with $|I| \ge p$ we have:
$$
\tr\left([\oR]_I^{-1}[A]_I\right) \ge 2\!\!\!\!\max\limits_{x \in \BR^{|I|}_+\setminus\{0\}}\frac{x'[\oR]_I^{-1}[A]_I[\oR]_I^{-1}x}{x'[\oR]_I^{-1}x}.
$$
Then the $\SRBM^d(R, \mu, A)$ avoids edges of order $p$. 

(iii) Suppose there exists a subset $I \subseteq \{1, \ldots, d\}$ such that 
$$
\tr\left([\oR]_I^{-1}[A]_I\right) < 2\!\!\!\!\min\limits_{x \in \BR^{|I|}_+\setminus\{0\}}\frac{x'[\oR]_I^{-1}[A]_I[\oR]_I^{-1}x}{x'[\oR]_I^{-1}x}.
$$
Then the $\SRBM^d(R, \mu, A)$ hits $S_I$. 
\end{cor}

\section{Proofs of Theorems~\ref{totalcor}, ~\ref{cams} and~\ref{mainthm}}

\subsection{Outline of the proofs}

Consider a system of competing Brownian particles from Definition~\ref{classicdef}. In Lemma~\ref{red}, we note that a multicollision with pattern $I$ is equivalent to an $\SRBM^{N-1}(R, \mu, A)$ hitting the edge $S_I$ of the $N-1$-dimensional orthant $\BR^{N-1}_+$. Here, the parameters $R$, $\mu$, $A$ are given by~\eqref{R12}, ~\eqref{mu} and~\eqref{A} below. We apply Corollary~\ref{cor:inv-R} and Theorem~\ref{corner2edge} to this SRBM to prove Theorems~\ref{totalcor}  and~\ref{mainthm} respectively.  We use the estimate in Lemma~\ref{simple} for $c_+$, since the right-hand side of~\eqref{Ncorner} seems hard to compute for matrices $R$ and $A$ given by~\eqref{R12} and~\eqref{A}. 

Note that the matrix $R$ from~\eqref{R12} is itself symmetric. Therefore, in Corollary~\ref{cor:inv-R} we can take $C = I_{N-1}$ and $\ol{R} = R$. The inverse matrix $R^{-1} = \ol{R}^{-1} = (\rho_{ij})_{1 \le i, j \le N-1}$ has the form
\begin{equation}
\label{invR}
\rho_{ij} = 
\begin{cases}
2i(N-j)/N,\ i \le j;\\
2j(N-i)/N,\ i \ge j
\end{cases}
\end{equation}
This result can be found in \cite{FP2001, Inversion} (the latter article deals with a slightly different matrix, from which one can easily find the inverse of the given matrix $R$). After a (rather tedious) computation, we rewrite the condition~\eqref{Ncorner-om} from Corollary~\ref{cor:inv-R} as $\CP(\si) \ge 0$, where $\CP(\si)$ is defined in~\eqref{CP}. This proves Theorem~\ref{totalcor}. 

Proving Theorem~\ref{mainthm} is a bit harder. Apply Theorem~\ref{corner2edge}, and fix a subset $I \subseteq \{1, \ldots, N-1\}$ such that $J \subseteq I$. We need to find a sufficient condition for an $\SRBM^{|I|}([R]_I, [\mu]_I, [A]_I)$ to a.s. avoid the corner of the orthant $\BR^{|I|}_+$. We decompose the set $I$ as in~\eqref{decomp}:
$$
I = I_1\cup I_2\cup \ldots I_r,
$$
into a union of disjoint non-adjacent discrete intervals. In Lemma~\ref{cornerstone}, we prove that if $I$ satisfies Assumption (B), then the $\SRBM^{|I|}([R]_I, [\mu]_I, [A]_I)$ indeed avoids the corner. This completes the proof of Theorem~\ref{mainthm}. But to prove Lemma~\ref{cornerstone}, we need to consider different variants of decomposition~\eqref{decomp}. For example, if $I_1 = \{1\}$ and $I_2 = \{3\}$, then this guarantees that an $\SRBM^{|I|}([R]_I, [\mu]_I, [A]_I)$ avoids the corner. Various cases are considered in Lemmas~\ref{61}, ~\ref{62} and~\ref{63}, which constitute the crux of the proof.

\subsection{Connection between an SRBM and competing Brownian particles}

Let us reduce multiple collisions of competing Brownian particles to an SRBM hitting edges of the boundary of high order. Consider the classical system of competing Brownian particles from Definition~\ref{classicdef}. By definition, the ranked particles $Y_1, \ldots, Y_N$ satisfy
$$
Y_1(t) \le \ldots \le Y_N(t).
$$
Consider the {\it gap process}: an $\BR^{N-1}_+$-valued process defined by
$$
Z = (Z(t), t \ge 0),\ \ Z(t) = (Z_1(t), \ldots, Z_{N-1}(t))', \ \ Z_k(t) = Y_{k+1}(t) - Y_k(t).
$$
It was shown in \cite{BFK2005} that this is an $\SRBM^{N-1}(R, \mu, A)$ in the orthant $S = \BR_+^{N-1}$ with parameters 
\begin{equation}
\label{R12}
R = 
\begin{bmatrix}
1 & -1/2 & 0 & 0 & \ldots & 0 & 0\\
-1/2 & 1 & -1/2 & 0 & \ldots & 0 & 0\\
0 & -1/2 & 1 & 0 & \ldots & 0 & 0\\
\vdots & \vdots & \vdots & \vdots & \ddots & \ddots & \ddots\\
0 & 0 & 0 & 0 & \ldots & 1 & -1/2\\
0 & 0 & 0 & 0 & \ldots & -1/2 & 1
\end{bmatrix},
\end{equation}
\begin{equation}
\label{mu}
\mu = \left(g_2 - g_1, g_3 - g_4, \ldots, g_N - g_{N-1}\right)',
\end{equation}
\begin{equation}
\label{A}
A = 
\begin{bmatrix}
\si_1^2 + \si_2^2 & -\si_2^2 & 0 & 0 & \ldots & 0 & 0\\
-\si_2^2 & \si_2^2 + \si_3^2 & -\si_3^2 & 0 & \ldots & 0 & 0\\
0 & -\si_3^2 & \si_3^2 + \si_4^2 & -\si_4^2 & \ldots & 0 & 0\\
\vdots & \vdots & \vdots & \vdots & \ddots & \vdots & \vdots\\
0 & 0 & 0 & 0 & \ldots & \si_{N-2}^2 + \si_{N-1}^2 & -\si_{N-1}^2\\
0 & 0 & 0 & 0 & \ldots & -\si_{N-1}^2 & \si_{N-1}^2 + \si_N^2
\end{bmatrix}
\end{equation}
Note that the matrix $R$ is a reflection nonsingular $\CM$-matrix. This follows from the fact that $I_{N-1} - R \ge 0$, and $R^{-1} \ge 0$ (which, in turn, was proved in \cite[Proposition 2.1(i)]{MyOwn6}). The following lemma translates statements about multiple collisions and multicollisions of competing Brownian particles to the language of an SRBM. The proof is trivial and is therefore omitted. 

\begin{lemma}
\label{red}
Consider a classical system of $N$ competing Brownian particles from Definition~\ref{classicdef}. Then there is a multicollision with pattern $I$ at time $t$ if and only if the gap process hits the edge $S_I$ at time $t$. For example, there is a total collision at time $t$ if and only if the gap process hits the corner at time $t$. 
\end{lemma}

For example, $Y_1(t) = Y_2(t)$ and $Y_3(t) = Y_4(t) = Y_5(t)$ is a multicollision of order $3$, with pattern $\{1, 3, 4\}$, which is equivalent of the gap process hitting the edge $\{z_1 = z_3 = z_4 = 0\}$. Similarly, $Y_3(t) = Y_4(t) = Y_5(t) = Y_6(t)$ is a collision of order $3$ (which is also a particular case of a multicollision of order $3$, with pattern $\{3, 4, 5\}$), and it is equivalent to the gap process hitting the edge $\{z_3 = z_4 = z_5 = 0\}$. 

\subsection{Some preliminary calculations}

As mentioned before, the matrix $R$ in~\eqref{R12} is itself symmetric. Therefore, we can take $C = I_{N-1}$, and $\oR = R$. Without loss of generality, let 
$$
\rho_{ij} = 0,\ \ i = 0,\ N,\ j = 0, \ldots, N\ \mbox{or}\ j = 0,\ N,\ i = 0, \ldots, N.
$$
This is consistent with the notation~\eqref{invR}. Note that $\rho_{ij} > 0$ for $i, j = 1, \ldots, N-1$: all elements of the matrix $R^{-1}$ are positive. Therefore, we can apply an estimate from Lemma~\ref{simple}:
$$
c_+ := \max\limits_{x \in \BR^{N-1}\setminus\{0\}}
\frac{x'R^{-1}AR^{-1}x}{x'R^{-1}x} \le 
\max\limits_{1 \le k \le l \le N-1}\frac{\bigl(R^{-1}AR^{-1}\bigr)_{kl}}{\rho_{kl}}.
$$

\begin{lemma}
\label{trace}
For the matrix $R$ given by~\eqref{R12} and the matrix $A$ given by~\eqref{A}, we have in the notation of ~\eqref{CT}:
\begin{equation}
\label{CTcalc}
\tr\bigl(R^{-1}A\bigr) = \CT(\si)\,.
\end{equation}
\end{lemma}

\begin{proof} Straightforward calculation gives
\begin{align*}
\tr\bigl(R^{-1}A\bigr) =& \SL_{i=1}^{N-1}\SL_{j=1}^{N-1}\rho_{ij}a_{ij} = \SL_{i=1}^{N-1}(\si_i^2 + \si_{i+1}^2)\frac{2i(N-i)}N   \\ & + 2\SL_{i=2}^{N-1}(-\si_i^2)\frac{2(i-1)(N-i)}N = \frac{2(N-1)}N\si_1^2 + \frac{2(N-1)}N\si_N^2  \\ & + \SL_{k=2}^{N-1}\si_k^2\left(\frac{2k(N-k)}N + \frac{2(k-1)(N-k+1)}N - 2\frac{2(k-1)(N-k)}N\right) \\ & = \frac{2(N-1)}N\SL_{k=1}^N\si_k^2 = \CT(\si). 
\end{align*}
\end{proof}

The following lemma helps us simplify the matrix $R^{-1}AR^{-1}$, where $A$ is given by~\eqref{A}, and $R^{-1}$ is given by~\eqref{invR}.  

\begin{lemma} Consider the matrix $A$ as in~\eqref{A}, and take a symmetric $(N-1)\times (N-1)$-matrix $Q = (q_{ij})$. Augment it by two additional rows and two additional columns, one from each side, and fill them with zeros: 
$$
q_{ij} = 0\ \ \mbox{for}\ \ i = 0,\, N,\ j = 0, \ldots, N, \  \mbox{and for}\ \ j = 0,\, N,\ i = 0, \ldots, N.
$$
Then for $k, l = 1, \ldots, N-1$ we have:
$$
(QAQ)_{kl} = \SL_{p=1}^N\left(q_{pk} - q_{p-1, k}\right)\left(q_{pl} - q_{p-1, l}\right)\si_p^2.
$$
\label{matrixes}
\end{lemma}

\begin{proof} The matrix $A$ is tridiagonal: 
$$
\begin{cases}
a_{ii} = \si_i^2 + \si_{i+1}^2,\ i = 1, \ldots, N-1;\\
a_{i, i+1} = a_{i+1, i} = -\si_{i+1}^2,\ i = 1, \ldots, N-2;\\
a_{ij} = 0,\ \ i, j = 1, \ldots, N-1,\ |i - j| \ge 2.
\end{cases}
$$
Using the symmetry of $Q$, we have: 
\begin{align*}
(QAQ)_{kl} =& \SL_{i=1}^{N-1}\SL_{j=1}^{N-1}q_{ik}q_{jl}a_{ij} = \SL_{p=1}^{N-1}\left(\si_p^2 + \si_{p+1}^2\right)q_{pk}q_{pl} - \SL_{p=2}^{N-1}\si_p^2q_{pk}q_{p-1, l} - \SL_{p=2}^{N-1}\si_p^2q_{p-1, k}q_{pl} \\ & = 
\SL_{p=1}^{N}\si_p^2q_{pk}q_{pl} + \SL_{p=1}^{N}\si_p^2q_{p-1, k}q_{p-1, l} - \SL_{p=1}^{N}\si_p^2q_{pk}q_{p-1, l} - \SL_{p=1}^{N}\si_p^2q_{p-1, k}q_{pl} \\ & = 
\SL_{p=1}^{N}\left(q_{pk} - q_{p-1, k}\right)\left(q_{pl} - q_{p-1, l}\right)\si_p^2.
\end{align*}
\end{proof}

Lemma~\ref{matrixes} enables us to calculate $(R^{-1}AR^{-1})_{kl}$, where $A$ and $R$ are given by~\eqref{A} and~\eqref{R12}. 

\begin{lemma} Suppose the matrix $R$ is given by~\eqref{R12}, and the matrix $A$ is given by~\eqref{A}. 
Then for $1 \le k \le l \le N-1$ we have:
\begin{equation}
\label{RAR}
\bigl(R^{-1}AR^{-1}\bigr)_{kl} = \frac{4(N-k)(N-l)}{N^2}\SL_{p=1}^k\si_p^2 - \frac{4k(N-l)}{N^2}\SL_{p=k+1}^l\si_p^2 
+ \frac{4kl}{N^2}\SL_{p=l+1}^N\si_p^2. 
\end{equation}
\label{RARlemma}
\end{lemma}

\begin{proof} Apply Lemma~\ref{matrixes} to $Q = R^{-1}$, given by~\eqref{invR}, so that $q_{ij} = \rho_{ij}$. For $p \le k$, we get: 
For $p \le k$ we have:
$$
\rho_{pk} - \rho_{p-1, k} = \frac{2p(N-k)}{N} - \frac{2(p-1)(N-k)}N = \frac{2(N-k)}N,
$$
$$
\rho_{pl} - \rho_{p-1, l} = \frac{2p(N-l)}{N} - \frac{2(p-1)(N-l)}N = \frac{2(N-l)}N.
$$
For $k < p \le l$, we have:
$$
\rho_{pk} - \rho_{p-1, k} = \frac{2k(N-p)}{N} - \frac{2k(N-p+1)}N = -\frac{2k}N,
$$
$$
\rho_{pl} - \rho_{p-1, l} = \frac{2p(N-l)}{N} - \frac{2(p-1)(N-l)}N = \frac{2(N-l)}N.
$$
For $p > l$, we have:
$$
\rho_{pk} - \rho_{p-1, k} = \frac{2p(N-k)}{N} - \frac{2(p-1)(N-k)}N = \frac{2(N-k)}N,
$$
$$
\rho_{pl} - \rho_{p-1, l} = \frac{2p(N-l)}{N} - \frac{2(p-1)(N-l)}N = \frac{2(N-l)}N.
$$
The rest of the proof is trivial. 
\end{proof}

\subsection{Proof of Theorem~\ref{totalcor}} 

Use Corollary~\ref{cor:inv-R} and Corollary~\ref{simple} for matrices $R$ and $A$, given by~\eqref{R12} and~\eqref{A} respectively. We have the following sufficient condition for avoiding total collisions:
\begin{equation}
\label{form}
\tr\bigl(R^{-1}A\bigr) - 2\max\limits_{1 \le k \le l \le N-1}\frac{(R^{-1} AR^{-1})_{kl}}{\rho_{kl}} \ge 0.
\end{equation}
For $1 \le k \le l \le N-1$, denote 
$$
c_{k, l}(\si) = \tr\bigl(R^{-1}A\bigr) -  2\frac{\bigl(R^{-1}AR^{-1}\bigr)_{kl}}{\rho_{kl}}.
$$
Then we have:
\begin{equation}
\label{3}
\tr\bigl(R^{-1}A\bigr) - 2\max\limits_{k, l = 1, \ldots, N-1}\frac{(R^{-1} AR^{-1})_{kl}}{\rho_{kl}} = \min\limits_{1 \le k \le l \le N-1}c_{k, l}(\si).
\end{equation}

\begin{lemma} Using definitions of $c_l(\si)$ and $\si^{\leftarrow}$ from subsection 1.2, we have:

(i) For $2 \le k \le l \le N-2$, we have: $c_{k, l}(\si) \ge 0$.

(ii) For $1 = k \le l \le N - 1$, we have: $c_{k, l}(\si) = c_l(\si)$. 

(iii) For $1 \le k \le l = N-1$, we have: $c_{k, l}(\si) = c_{N-k}\left(\si^{\leftarrow}\right)$. 

\label{5555}
\end{lemma}

Assuming that Lemma~\ref{5555} is proved, let us finish the proof of Theorem~\ref{totalcor}. Let
\begin{equation}
\label{4}
\de(\si) := \min\limits_{2 \le k \le l \le N-2}c_{k, l}(\si).
\end{equation}
If $N < 4$, let $\de(\si) := 0$. By Lemma~\ref{5555} (i), we always have: $\de(\si) \ge 0$. Recall the definition of $\CP(\si)$ from~\eqref{CP} and use Lemma~\ref{5555} (ii), (iii):
\begin{equation}
\label{5}
\min\left(c_{1, 1}(\si),\, c_{1, 2}(\si),\, \ldots,\, c_{1, N-1}(\si),\, c_{2, N-1}(\si),\, \ldots,\, c_{N-1, N-1}(\si)\right) = \CP(\si).
\end{equation}
Comparing~\eqref{3}, ~\eqref{4} and~\eqref{5}, we have:
\begin{equation}
\label{10000}
\min\limits_{1 \le k \le l \le N-1}\left[\tr\bigl(R^{-1}A\bigr) - 2\frac{(R^{-1} AR^{-1})_{kl}}{\rho_{kl}}\right] = \min(\CP(\si), \de(\si)).
\end{equation}
Thus
$$
\min\limits_{1 \le k \le l \le N-1}c_{k, l}(\si) \ge 0\ \ \mbox{if and only if}\ \ \CP(\si) \ge 0.
$$
This completes the proof of Theorem~\ref{totalcor}.  \qed

\medskip
\noindent

{\it Proof of Lemma~\ref{5555}}: We can simplify the expression for $c_{k, l}(\si)$. Applying~\eqref{RAR} and~\eqref{invR}, we have: for $1 \le k \le l \le N -1$,
$$
\frac{\bigl(R^{-1}AR^{-1}\bigr)_{kl}}{\rho_{kl}} = 
\frac{2(N-k)}{Nk}\SL_{p=1}^k\si_p^2 - \frac{2}{N}\SL_{p=k+1}^l\si_p^2 
+ \frac{2l}{N(N-l)}\SL_{p=l+1}^N\si_p^2.
$$
Therefore, we have:
\begin{align*}
c_{k, l}(\si) := & \left(\frac{2(N-1)}{N} - \frac{4(N-k)}{Nk}\right)\SL_{p=1}^k\si_p^2 \\ & + \left(\frac{2(N-1)}{N} + \frac{4}{N}\right)\SL_{p=k+1}^l\si_p^2
+ \left(\frac{2(N-1)}{N} - \frac{4l}{(N-l)N}\right)\SL_{p=l+1}^N\si_p^2
\\ & = \frac{2(N-1)k - 4(N-k)}{kN}\SL_{p=1}^k\si_p^2  + \frac{2(N+1)}N\SL_{p=k+1}^l\si_p^2 \\ &+ \frac{2(N-1)(N-l) - 4l}{(N-l)N}\SL_{p=l+1}^N\si_p^2.
\end{align*}
Now, for $k \ge 2$ we get: 
$$
2(N-1)k - 4(N - k) \ge 4(N-1) - 4N + 8 = 4 \ge 0.
$$
Similarly, for $l \le N-2$ we get:
$$
2(N-1)(N-l) - 4l \ge 0.
$$
This proves part (i) of Lemma~\ref{5555}. Parts (ii) and (iii) are now straightforward. \ \  $\square$

\subsection{Proof of Theorem~\ref{mainthm}} Fix a subset $I \subseteq \{1, \ldots, N-1\}$ such that $J \subseteq I$. Take the matrices $R$ and $A$ given by~\eqref{R12} and~\eqref{A}. Essentially, we need to prove the following lemma:

\begin{lemma}
\label{cornerstone}
If the subset $I$ satisfies Assumption (B), then the process 
$$
Z = (Z(t), t \ge 0) = \SRBM^{|I|}\left([R]_I, 0, [A]_I\right)
$$
does not hit the origin.
\end{lemma}

If we prove Lemma~\ref{cornerstone}, then Theorem~\ref{mainthm} will automatically follow from this lemma and Theorem~\ref{corner2edge}. 
The rest of this subsection is devoted to the proof of Lemma~\ref{cornerstone}. 

Let us investigate the structure of the matrices $[R]_I^{-1}$ and $[A]_I^{-1}$. Split $I$ into disjoint non-adjacent discrete intervals: $I = I_1\cup I_2\cup \ldots \cup I_r$. Since the matrices $R$ and $A$ are tridiagonal, the matrices $[R]_I$ and $[A]_I$ have the following block-diagonal form:
$$
[R]_I = \diag\left([R]_{I_1}, \ldots, [R]_{I_r}\right),\ \ [A]_I = \diag\left([A]_{I_1}, \ldots, [A]_{I_r}\right).
$$
The following processes are independent SRBMs:
\begin{equation}
\label{components00}
[Z]_{I_j} = ([Z(t)]_{I_j}, t \ge 0) = \SRBM^{|I_j|}\left([R]_{I_j}, 0, [A]_{I_j}\right),\ \ j = 1, \ldots, s.
\end{equation}
For any subset $I' = I_{i_1}\cup\ldots\cup I_{i_s}$, the process
$$
[Z]_{I'} = ([Z(t)]_{I'}, t \ge 0) = \SRBM^{|I'|}\left([R]_{I'}, 0, [A]_{I'}\right).
$$
\begin{rmk}
\label{rmkcrucial}
If for some choice of $I'$ this process does not hit the origin of $\BR^{|I'|}_+$, then the original process $Z$ does not hit the origin, because of independence of~\eqref{components00}. In particular, for each $j = 1, \ldots, s$, the process $[Z]_{I_j}$ does not hit the origin, then $Z$ does not hit the origin. 
\end{rmk}

Now, let us state three lemmata.

\begin{lemma} If at least two of the discrete intervals $I_1,\ldots, I_r$ are singletons, then $Z$ a.s. at any time $t > 0$ does not hit the origin. 
\label{61}
\end{lemma}

\begin{lemma} If at least one $I_1,\ldots, I_r$ is a two-element subset $\{k-1, k\}$ with local concavity at $k$, then $Z$ a.s. at any time $t > 0$ does not hit the origin. 
\label{62}
\end{lemma}

\begin{lemma} If $I$ satisfies Assumption (A), then $Z$ a.s. at any time $t > 0$  does not hit the origin. 
\label{63}
\end{lemma}

Combining Lemmas~\ref{61}, ~\ref{62}, and~\ref{63} with Remark~\ref{rmkcrucial}, we complete the proof of Lemma~\ref{cornerstone} and Theorem~\ref{mainthm}.  \qed

\medskip
In the remainder of this subsection, we shall prove these three lemmas. 

\medskip
\noindent
{\it Proof of Lemma~\ref{61}}: Without loss of generality, suppose $I_1 = \{k\}$ and $I_2 = \{l\}$ are singletons. Since they are not adjacent, $|k - l| \ge 2$; assume that $k < l$, so that $l \ge  k+2$. Then 
$$
\left(Z_k, Z_l\right)' = \SRBM^2\left([R]_{I_1\cup I_2}, 0, [A]_{I_1\cup I_2}\right).
$$
But 
$$
[A]_{I_1\cup I_2} = \begin{bmatrix} \si_k^2 + \si_{k+1}^2 & 0 \\
0 & \si_l^2 + \si_{l+1}^2\end{bmatrix},  \ \ 
[R]_{I_1\cup I_2} = I_2.
$$
Therefore, $Z_k$ and $Z_l$ are independent reflected Brownian motions on $\BR_+$. They do not hit zero simultaneously, which is the same as to say  that  $(Z_k, Z_l)'$ does not hit the origin in $\BR^2_+$. 

\medskip
\noindent
{\it Proof of Lemma~\ref{62}}: Assume without loss of generality that $I_1 = \{1, 2\}$, and we have local concavity at $2$: $\si_2^2 \ge (\si_1^2 + \si_3^2)/2$. By Remark~\ref{rmkcrucial}, it suffices to show that an $\SRBM^2([R]_{I_1}, [\mu]_{I_1}, [A]_{I_1})$ does not hit the origin. Because of the connection between an SRBM and systems of competing Brownian particles outlined in subsection 4.2, this, in turn, is equivalent of a system of three competing Brownian particles with diffusion coefficients $\si_1^2, \si_2^2, \si_3^2$ not having a triple collision. But this last statement follows from Proposition~\ref{elegant}, applied to the case $N = 3$. 

\medskip
\noindent
{\it Proof of Lemma~\ref{63}}: By \cite[Lemma 5.6]{MyOwn2}, the matrices $[R]_{I_1}, \ldots, [R]_{I_r}$ are themselves reflection nonsingular $\CM$-matrices. Therefore, they are invertible, and 
$$
[R]^{-1} = \diag\left([R]^{-1}_{I_1}, \ldots, [R]^{-1}_{I_r}\right).
$$
In addition,
\begin{equation}
\label{4568}
[R]_I^{-1}[A]_I^{-1} = \diag\left([R]^{-1}_{I_1}[A]_{I_1}, \ldots, [R]^{-1}_{I_r}[A]_{I_r}\right),
\end{equation}
$$
[R]_I^{-1}[A]_I^{-1}[R]_I^{-1} = \diag\left([R]^{-1}_{I_1}[A]_{I_1}[R]_{I_1}^{-1}, \ldots, [R]^{-1}_{I_r}[A]_{I_r}[R]_{I_r}^{-1}\right).
$$

\begin{lemma} For the matrices $R$ and $A$ given by~\eqref{R12} and~\eqref{A}, we have:
\begin{equation}
\label{11}
\tr\bigl([R]_I^{-1}[A]_I^{-1}\bigr) = \SL_{j=1}^r\CT(\ol{I}_j).
\end{equation}
\end{lemma}

\begin{proof}
Because of~\eqref{4568}, we get:
\begin{equation}
\label{21}
\tr\bigl([R]_I^{-1}[A]_I^{-1}\bigr) = \SL_{j=1}^r\tr\bigl([R]^{-1}_{I_j}[A]_{I_j}\bigr).
\end{equation}
Applying Lemma~\ref{trace} with $I_j$ instead of $\{1, \ldots, N-1\}$ and $\ol{I}_j$ instead of $\{1, \ldots, N\}$, $j = 1, \ldots, r$, we have:
\begin{equation}
\label{22}
\tr\bigl([R]^{-1}_{I_j}[A]_{I_j}\bigr) = \SL_{j=1}^r\CT(\ol{I}_j),\ \ j = 1, \ldots, r.
\end{equation}
Combining~\eqref{21} and~\eqref{22}, we get~\eqref{11}. 
\end{proof}

\begin{lemma}
We have the following estimate:
\begin{equation}
\label{100}
\max\limits_{x \in \BR^{|I|}_+\setminus\{0\}}\frac{x'[R]_I^{-1}[A]_I[R]_I^{-1}x}{x'[R]^{-1}_Ix} \le  \max\limits_{j = 1, \ldots, r}\max\limits_{\substack{k, l \in I_j\\ k \le l}}\frac{\bigl([R]_{I_j}^{-1}[A]_{I_j}[R]_{I_j}^{-1}\bigr)_{kl}}{\bigl([R]_{I_j}^{-1}\bigr)_{kl}}.
\end{equation}
\label{651}
\end{lemma}

The proof of Lemma~\ref{651} is postponed until the end of this section. Assuming we have proved it, let us show how to finish the proof of Lemma~\ref{63}. 

Using~\eqref{100} and~\eqref{11}, we can rewrite the condition~\eqref{Nedge} as
$$
\SL_{j=1}^r\CT(\ol{I}_j) - 2\!\!\max\limits_{i = 1, \ldots, r}\max\limits_{\substack{k, l \in I_i\\ k \le l}}\frac{\bigl([R]_{I_i}^{-1}[A]_{I_i}^{-1}[R]_{I_i}^{-1}\bigr)_{kl}}{\bigl([R]_{I_i}^{-1}\bigr)_{kl}} 
\ge 0.
$$
Equivalently,
$$
\SL_{\substack{j=1\\j \ne i}}^r\CT(\ol{I}_j) + \CT(\ol{I}_i) - 2\max\limits_{\substack{k, l \in I_i\\ k \le l}}\frac{\bigl([R]_{I_i}^{-1}[A]_{I_i}^{-1}[R]_{I_i}^{-1}\bigr)_{kl}}{\bigl([R]_{I_i}^{-1}\bigr)_{kl}} \ge 0,\ \ i = 1, \ldots, r.
$$
In the proof of Theorem~\ref{totalcor}, see~\eqref{10000} and~\eqref{CTcalc}, it was shown that for $i = 1, \ldots, r$, we have:
$$
\CT(\ol{I}_i) - 2\max\limits_{\substack{k, l \in I_i\\ k \le l}}\frac{\bigl([R]_{I_i}^{-1}[A]_{I_i}^{-1}[R]_{I_i}^{-1}\bigr)_{kl}}{\bigl([R]_{I_i}^{-1}\bigr)_{kl}} = 
\min(\CP(\ol{I}_i), \de_i),\ \ \de_i := \de([\si]_{\ol{I}_i}) \ge 0.
$$ 
Therefore, the condition~\eqref{Nedge} is equivalent to 
\begin{equation}
\label{newform}
\SL_{j\ne i}\CT(\ol{I}_j) + \min(\CP(\ol{I}_i), \de_i) \ge 0,\ \ i = 1, \ldots, r.
\end{equation}
The condition~\eqref{newform}, in turn, is equivalent to
$$
\SL_{j\ne i}\CT(\ol{I}_j) + \CP(\ol{I}_i) \ge 0,\ \ i = 1, \ldots, r.
$$
This completes the proof of Lemma~\ref{63}, and with it the proofs of Lemma~\ref{cornerstone} and Theorem~\ref{mainthm}. 

\medskip

{\it Proof of Lemma~\ref{651}.} The matrices $[R]_I^{-1}$ and $[A]_I^{-1}$ are block-diagonal, with the blocks corresponding to the sets $I_1, \ldots, I_r$ of indices. Therefore, 
\begin{equation}
\label{42}
x'[R]_I^{-1}[A]_I[R]_I^{-1}x = \SL_{j=1}^r[x]'_{I_j}[R]_{I_j}^{-1}[A]_{I_j}[R]_{I_j}^{-1}[x]_{I_j},\ \ x'[R]_I^{-1}x = \SL_{j=1}^r[x]'_{I_j}[R]_{I_j}^{-1}[x]_{I_j}.
\end{equation}
Let $\mathcal Q(x) := \{j = 1, \ldots, r\mid [x]_{I_j} \ne 0\}$. We might as well rewrite~\eqref{42} as
$$
x'[R]_I^{-1}[A]_I[R]_I^{-1}x =  \SL_{j \in \mathcal Q(x)}[x]'_{I_j}[R]_{I_j}^{-1}[A]_{I_j}[R]_{I_j}^{-1}[x]_{I_j},\ \  x'[R]_I^{-1}x = \SL_{j\in \mathcal Q(x)}[x]'_{I_j}[R]_{I_j}^{-1}[x]_{I_j}.
$$
For $j \in \mathcal Q(x)$, we have: $[x]_{I_j} \in S^{|I_j|}_+\setminus\{0\}$. The matrix $[R]_{I_j}$ has the same form as $R$ in~\eqref{R12}, but with smaller size. Therefore, all elements of the inverse matrix $[R]_{I_i}^{-1}$ (just like for $R^{-1}$) are positive. Therefore, 
$[x]'_{I_i}[R]_{I_i}^{-1}[x]_{I_i} > 0$, $i = 1, \ldots, r$. Applying Lemma~\ref{fraccomp} to $a_i = [x]'_{I_i}[R]_{I_i}^{-1}[A]_{I_i}[R]_{I_i}^{-1}[x]_{I_i}$ and $b_i = [x]'_{I_i}[R]_{I_i}^{-1}[x]_{I_i} > 0$ for $i \in \mathcal Q(x)$, we get:
\begin{equation}
\label{10001}
\frac{x'[R]_I^{-1}[A]_I[R]_I^{-1}x}{x'[R]_I^{-1}x} \le 
\max\limits_{j \in \mathcal Q(x)}\frac{[x]'_{I_j}[R]_{I_j}^{-1}[A]_{I_j}[R]_{I_j}^{-1}[x]_{I_j}}{[x]'_{I_j}[R]_{I_j}^{-1}[x]_{I_j}}.
\end{equation}
But the matrix $[R]_{I_j}$, as just mentioned, has all elements positive. Applying Lemma~\ref{simple}, we have for $y \in \BR^{|I_j|}_+\setminus\{0\}$:
\begin{equation}
\label{10002}
\frac{y'[R]_{I_j}^{-1}[A]_{I_j}[R]_{I_j}^{-1}y}{y'[R]_{I_j}^{-1}y} \le 
\max\limits_{\substack{k, l \in I_j\\ k \le l}}\frac{\left([R]_{I_j}^{-1}[A]_{I_j}[R]_{I_j}^{-1}\right)_{kl}}{\left([R]_{I_j}^{-1}\right)_{kl}}.
\end{equation}
Combining~\eqref{10001} and~\eqref{10002}, we get~\eqref{100}. 
$\square$

\subsection{Proof of Theorem~\ref{cams}}

Recall the setting of Theorem~\ref{cornerthm}: we have a process $Z = (Z(t), t \ge 0)$ in $\BR^d_+$, which is an $\SRBM^d(R, \mu, A)$ with a reflection nonsingular $\CM$-matrix $R$. We would like this process  to avoid the corner $\{0\}$. We have: $d = N - 1 = 3$, and 
$$
R = 
\begin{bmatrix}
1 & -1/2 & 0\\
-1/2 & 1 & -1/2\\
0 & -1/2 & 1
\end{bmatrix},\ \ 
A = 
\begin{bmatrix}
\si_1^2 + \si_2^2 & -\si_2^2 & 0\\
-\si_2^2 & \si_2^2 + \si_3^2 & -\si_3^2\\
0 & -\si_3^2 & \si_3^2 + \si_4^2 
\end{bmatrix}
$$
We pick the following matrix:
\begin{equation}
\label{QDef}
Q = 
\begin{bmatrix}
1 & 1 & 1 \\
1 & \lambda & 1 \\
1 & 1 & 1 \\
\end{bmatrix},
\ \ \mbox{where}\ \ \lambda =\frac{\si_1^2+\si_2^2+\si_3^2+\si_4^2}{\si_2^2+\si_3^2}.
\end{equation}
This is a symmetric matrix. It is also strictly copositive, because all its elements are strictly positive. Let us show that conditions of Theorem~\ref{cornerthm} (i) hold. First, calculations show that 
\begin{equation}
QR=
\begin{bmatrix}
\frac{1}{2} & 0 & \frac{1}{2} \\
1-\frac{\lambda}{2} & \lambda-1 & 1-\frac{\lambda}{2} \\
\frac{1}{2} & 0 & \frac{1}{2}
\end{bmatrix},
\end{equation}
Therefore, $(QR)_{ij} \ge 0$ for $i \ne j$ is equivalent to 
$$
1-\frac{\lambda}{2}\ge 0 \iff \lambda \le 2 \iff \si_2^2+\si_3^2 \ge \si_1^2 +\si_4^2.
$$
By a simple calculation, one can also confirm the relation 
$$
QAQ=\frac{\tr(QA)}{2}Q,
$$
and $\tr\left(QA\right) \ge 2c_+(Q)$. This completes the proof of Theorem~\ref{cams}.

\section{Competing Brownian Particles with Asymmetric Collisions}

One can generalize the classical system of competing Brownian particles from Definition~\ref{classicdef} in many ways. Let us describe one of these generalizations.
For $k = 1, \ldots, N-1$, let
$$
L_{(k, k+1)} = (L_{(k, k+1)}(t), t \ge 0)
$$
be the semimartingale local time process at zero of the process $Z_k = Y_{k+1} - Y_k$. We shall call this the {\it collision local time} of the particles $Y_k$ and $Y_{k+1}$. For notational convenience, let $L_{(0, 1)}(t) \equiv 0$ and $L_{(N, N+1)}(t) \equiv 0$. 
Let
$$
B_k(t) = \SL_{i=1}^N\int_0^t1(\mP_s(k) = i)\md W_i(s),\ \ k = 1, \ldots, N,\ \ t \ge 0.
$$
It can be checked that $\langle B_k, B_l\rangle_t \equiv \de_{kl}t$; that is, $B_1, \ldots, B_N$ are i.i.d. standard Brownian motions. 
As shown in \cite{BFK2005, BG2008, Ichiba11}, \cite[Chapter 3]{IchibaThesis}, the ranked particles $Y_1, \ldots, Y_N$ have the following dynamics:
$$
Y_k(t) = Y_k(0) + g_kt + \si_kB_k(t) - \frac12 L_{(k, k+1)}(t) + \frac12 L_{(k-1, k)}(t),\ \ k = 1, \ldots, N.
$$
The collision local time $L_{(k, k+1)}$ has a physical meaning of the push exerted when the particles $Y_k$ and $Y_{k+1}$ collide, which is needed to keep the particle $Y_{k+1}$ above the particle $Y_k$. Note that the coefficients at the local time terms are $\pm 1/2$. This means that the collision local time $L_{(k, k+1)}$ is split evenly between the two colliding particles: the lower-ranked particle $Y_k$ receives one-half of this local time, which pushes it down, and the higher-ranked particle $Y_{k+1}$ receives the other one-half of this local time, which pushes it up. 

In the paper \cite{KPS2012}, they considered systems of Brownian particles when this collision local time is split unevenly: the part $q^+_{k+1}L_{(k, k+1)}(t)$ goes to the upper particle $Y_{k+1}$, and the part $q^-_kL_{(k, k+1)}(t)$ goes to the lower particle $Y_k$. Let us give a formal definition.

\begin{defn} Fix $N \ge 2$, the number of particles. Take drift and diffusion coefficients
$$
g_1, \ldots, g_N;\ \ \si_1, \ldots, \si_N > 0,
$$
and, in addition, take {\it parameters of collision}
$$
q^{\pm}_1,\ldots, q^{\pm}_N \in (0, 1),\ \ q^+_{k+1} + q^-_k = 1,\ \ k = 1, \ldots, N-1.
$$
Consider a continuous adapted $\BR^N$-valued process
$$
Y = \left(Y(t) = (Y_1(t), \ldots, Y_N(t))', t \ge 0\right).
$$
Take other $N-1$ continuous adapted real-valued nondecreasing processes
$$
L_{(k, k+1)} = (L_{(k, k+1)}(t), t \ge 0),\ \ k = 1, \ldots, N-1,
$$
with $L_{(k, k+1)}(0) = 0$, which can increase only when $Y_{k+1} = Y_k$:
$$
\int_0^{\infty}1(Y_{k+1}(t) > Y_k(t))\md L_{(k, k+1)}(t) = 0,\ \ k = 1, \ldots, N-1.
$$
Let $L_{(0, 1)}(t) \equiv 0$ and $L_{(N, N+1)}(t) \equiv 0$. Assume that
\begin{equation}
\label{asymm}
Y_k(t) = Y_k(0) + g_kt + \si_kB_k(t) - q^-_kL_{(k, k+1)}(t) + q^+_kL_{(k-1, k)}(t),\ \ k = 1, \ldots, N.
\end{equation}
Then the process $Y$ is called the {\it system of competing Brownian particles with asymmetric collisions}. The gap process is defined similarly to the case of a classical system. 
\label{defn:asymm}
\end{defn}

Strong existence and pathwise uniqueness for these systems were shown in \cite[Section 2.1]{KPS2012}. When $q^{\pm}_1 = q^{\pm}_2 = \ldots = 1/2$, we are back in the case of symmetric collisions. 

\begin{rmk}
For systems of competing Brownian particles with asymmetric collisions,
we defined only ranked particles $Y_1, \ldots, Y_N$. It is, however, possible to define named particles $X_1, \ldots, X_N$ for the case of asymmetric collisions. This is done in \cite[Section 2.4]{KPS2012}. The construction works up to the first moment of a triple collision. A necessary and sufficient condition for a.s. absence of triple collisions is given in \cite{MyOwn3}. We will not make use of this construction in our article, instead working with ranked particles. 
\end{rmk}

We can define collisions and multicollisions similarly to the classical case, as in Definition~\ref{Pat}. It was shown in \cite{KPS2012} that the gap process for systems with asymmetric collisions, much like for the classical case, is an SRBM. Namely, it is an $\SRBM^{N-1}(R, \mu, A)$, where $\mu$ and $A$ are given by~\eqref{mu} and~\eqref{A}, and the reflection matrix $R$ is given by 
\begin{equation}
\label{R}
R = 
\begin{bmatrix}
1 & -q^-_2 & 0 & 0 & \ldots & 0 & 0\\
-q^+_2 & 1 & -q^-_3 & 0 & \ldots & 0 & 0\\
0 & -q^+_3 & 1 & -q^-_4 & \ldots & 0 & 0\\
\vdots & \vdots & \vdots & \vdots & \ddots & \vdots & \vdots\\
0 & 0 & 0 & 0 & \ldots & 1 & -q^-_{N-1}\\
0 & 0 & 0 & 0 & \ldots & -q^+_{N-1} & 1
\end{bmatrix}
\end{equation}
The connection between multicollisions and multiple collisions in this system and hitting of edges of $\BR^{N-1}_+$ by the gap process is the same as in Lemma~\ref{red}. This allows us to apply Theorem~\eqref{cornerthm} and Theorem~\eqref{corner2edge} to find sufficient conditions for avoiding multicollisions of a given pattern. In particular, the results of Lemma~\ref{Indep} are still valid for system with asymmetric collisions: the property of a.s. avoiding multicollisions of a certain pattern depends only on the diffusion coefficients and parameters of collision. 

A remark is in order: the matrix $R$ in~\eqref{R} in general is not symmetric, as opposed to the matrix $R$ in~\eqref{R12}. But if we take the $(N-1)\times(N-1)$ diagonal matrix
$$
C = \diag\left(1, \frac{q^+_2}{q^-_2}, \frac{q^+_2q^+_3}{q^-_2q^-_3}, \ldots, \frac{q^+_2q^+_3\ldots q^+_{N-1}}{q^-_2q^-_3\ldots q^-_{N-1}}\right),
$$
then the matrix $\ol{R} = RC$ is diagonal.

%
%
%
%

\section{Appendix}

\subsection{Proof of Lemma~\ref{Indep}} Follows from \cite[Lemma 3.1]{MyOwn3}, the discussion in \cite[Subsection 3.2]{MyOwn3}, and the reduction of multicollisions to hitting edges of the orthant which is done in Lemma~\ref{red} in this article. 

\subsection{Proof of Lemma~\ref{constants}}

Fix $x \in S\setminus\{0\}$. Since $Q$ is strictly copositive, we have: $x'Qx > 0$. Since $Q$ is nonsingular, $Qx \ne 0$. Since $A$ is positive definite, we have: $x'QAQx = (Qx)'A(Qx) > 0$. Therefore, the function
$$
f(x) := \frac{x'QAQx}{x'Qx}
$$
is well-defined and strictly positive on $S\setminus\{0\}$. In addition, it is {\it homogeneous}, in the sense that for $x \in S\setminus\{0\}$ and $k > 0$ we have: $f(kx) = f(x)$. Therefore, 
$$
\{f(x)\mid x \in S\setminus\{0\}\} = \{f(x)\mid x \in S,\ \norm{x} = 1\}.
$$
The set $\{x \in S\mid \norm{x} = 1\}$ is compact, and $f$ is continuous and positive on this set. Therefore, it is bounded on this set (and therefore on the set $S\setminus\{0\}$), and reaches its maximal and minimal values, both of which are strictly positive. 


\subsection{Proof of Corollary~\ref{cor:inv-R}} By \cite[Lemma 2.1]{MyOwn3} (equivalent characterization of reflection nonsingular $\CM$-matrices), we have: 
$$
\left(R^{-1}\right)_{ij} \ge 0,\ \ i, j = 1, \ldots, d; \ \ \left(R^{-1}\right)_{ii} > 0,\ \ i = 1, \ldots, d.
$$
Therefore, the matrix $\om = C^{-1}R^{-1} = (\rho_{ij})_{1 \le i, j \le d}$
has elements $\rho_{ij} = c_i^{-1}\left(R^{-1}\right)_{ij}$. By assumptions, the matrix $\om$ is symmetric. Therefore, its entries satisfy
\begin{equation}
\label{properties}
\rho_{ij}  = \rho_{ji} \ge 0,\ \ i, j = 1, \ldots, d; \ \ \rho_{ii} > 0,\ \ i = 1, \ldots, d.
\end{equation}
From here, it is easy to see that $\om$ is strictly copositive: $x'\om x > 0$ for $x \in S\setminus\{0\}$. Also, $(\om R)_{ij} = (C^{-1})_{ij} = 0$ for $i \ne j$. It suffices to apply Theorem~\ref{cornerthm}. 

\subsection{Proof of Lemma~\ref{simple}}

Let us prove the statement for the maximum. For the minimum, the proof is similar. For $x \in S\setminus\{0\}$, we have: $x_1, \ldots, x_d \ge 0$, and 
$$
\frac{x'\om A\om x}{x'\om x} = \frac{\sum_{i=1}^d\sum_{j=1}^d(\om A\om)_{ij}x_ix_j}{\sum_{i=1}^d\sum_{j=1}^d\rho_{ij}x_ix_j}.
$$
Apply Lemma~\ref{fraccomp} to $s = d^2$, $a_{ij} = (\om A\om)_{ij}x_ix_j$, $b_{ij} = \rho_{ij}x_ix_j$ (we index $a_i$ and $b_i$ by double indices, with each of the two indices ranging from $1$ to $d$). It suffices to note that, because of the symmetry of $\om A\om$ and $\om = (\rho_{ij})$, we have:
$$
\max\limits_{i, j = 1, \ldots, d}\frac{(\om A\om)_{ij}}{\rho_{ij}} = \max\limits_{1 \le i \le j \le d}\frac{(\om A\om)_{ij}}{\rho_{ij}}.
$$

\subsection{Miscellaneous lemmata}

\begin{lemma}
Take real numbers $a_1, \ldots, a_s$ and positive real numbers $b_1, \ldots, b_s$. Then 
$$
\min\left(\frac{a_1}{b_1}, \ldots, \frac{a_s}{b_s}\right) \le \frac{a_1 + \ldots + a_s}{b_1 + \ldots + b_s} \le \max\left(\frac{a_1}{b_1}, \ldots, \frac{a_s}{b_s}\right).
$$
\label{fraccomp}
\end{lemma}

\begin{proof}
Let us prove the inequality
$$
\frac{a_1 + \ldots + a_s}{b_1 + \ldots + b_s} \le \max\left(\frac{a_1}{b_1}, \ldots, \frac{a_s}{b_s}\right).
$$
The other inequality is proved similarly. Assume the converse: that 
$$
\frac{a_1 + \ldots + a_s}{b_1 + \ldots + b_s} > \frac{a_i}{b_i},\ \ i = 1, \ldots, s.
$$
Multiply the $i$th inequality by $(b_1 + \ldots + b_s)b_i > 0$:
$$
(a_1 + \ldots + a_s)b_i > a_i(b_1 + \ldots + b_s),\ \ i = 1, \ldots, s.
$$
Add them up and get:
$$
(a_1 + \ldots + a_s)(b_1 + \ldots + b_s) > (a_1 + \ldots + a_s)(b_1 + \ldots + b_s),
$$
and we arrive at a contradiction. 
\end{proof}

\begin{lemma} 
\label{lemma:aux-comparison}
Suppose we are given the following:

(i) two real-valued continuous adapted processes 
$$
X_1 = (X_1(t), t \ge 0)\ \ \mbox{and}\ \ X_2 = (X_2(t), t \ge 0),
$$
starting from the same $X_1(0) = X_2(0) = x$;

(ii) a real continuous function $\si : \BR \to \BR$ such that 
$$
|\si(x) - \si(y)| \le \rho(|x-y|),\ \ x, y \in \BR,\ t \ge 0,
$$
where $\rho : \BR_+ \to \BR_+$ is an increasing function such that $\rho(0) = 0$ and
$\int_0^{\infty}\rho^{-2}(s)\md s = \infty$;

(iii) a continuous function $b : \BR \to \BR$ and a continuous adapted process $\be = (\be(t), t \ge 0)$ with bounded variation, such that for every subset $A \subseteq \BR_+$, we have:
\begin{equation}
\label{eq:comp-measures}
\int_A\md\be(t) \ge \int_Ab(X_1(t))\md t,
\end{equation}
and the following equations are satisfied:
$$
\md X_1(t) = b(X_1(t))\md t + \si(X_1(t))\md W(t), \ \ \md X_2(t) = \md\be(t) + \si(X_2(t))\md W(t).
$$
Here, $W = (W(t), t \ge 0)$ is a standard Brownian motion. Then a.s. for all $t \ge 0$ we have: $X_1(t) \le X_2(t)$.  
\end{lemma}

\begin{proof} This is a modification of the proof of \cite[Theorem 6.1]{IWBook} and \cite[Theorem 1.1]{IkedaComp}. From the property~\eqref{eq:comp-measures}, we get: for any measurable function $\phi : \BR_+ \to \BR_+$ and any $t > 0$, we get:
$$
\int_0^t\phi(s)\md\be(s) \ge \int_0^t\phi(s)b(X_2(s))\md s.
$$
In the proof \cite[Theorem 1.1]{IkedaComp}, we should change $\be_2(s)\md s$ to 
$\md\be(s)$ and $\be_1(s)\md s$ to $b(X_1(s))\md s$. The rest of the proof should be modified accordingly. 
\end{proof}

\section*{Acknoweldgements}

The authors would like to thank \textsc{Ioannis Karatzas}, \textsc{Soumik Pal} and \textsc{Ruth Williams}, as well as an anonymous referee, for help and useful discussion. This research was partially supported by NSF grants DMS 1007563, DMS 1308340, DMS 1409434,  and DMS 1405210.


\bibliographystyle{plain}

\bibliography{aggregated}

\end{document}